\documentclass[11pt]{article}

\usepackage{mathtools}
\usepackage{amsthm,amssymb}
\usepackage{tikz,environ}
\usetikzlibrary{decorations.pathmorphing}
\usetikzlibrary{intersections}
\allowdisplaybreaks

\theoremstyle{plain}
\newtheorem{thm}{Theorem}
\newtheorem{lem}[thm]{Lemma}
\newtheorem{prop}[thm]{Proposition}

\theoremstyle{definition}
\newtheorem{example}[thm]{Example}
\newtheorem{defn}[thm]{Definition}

\newtheorem{conj}[thm]{Conjecture}
\newtheorem{quest}{Question}

\newcommand{\vmn}[1]{m_v(#1)} 
\newcommand{\mn}[1]{m(#1)} 

\newcommand{\meq}[2][0mm]{\raisebox{-.88mm}{\hspace{#1}\ensuremath{#2}\hspace{#1}}} 
\newcommand{\tile}[2][1]{\begin{mosaic}[#1] #2 \\ \end{mosaic}}

\NewEnviron{mosaic}[1][1]{\begin{tikzpicture}[baseline]\def\tilesize{#1}\matrix[mosai,ampersand replacement=\&]{\BODY};\end{tikzpicture}}

\NewEnviron{vmosaic}[7][1]{\begin{tikzpicture}[baseline]\def\tilesize{#1}\matrix[mosai,ampersand replacement=\&]{\BODY};
   \foreach \x [count=\y] in #4
      \node[above,yshift=#1*.5cm-1] at (-#1*.5-#1*#3*.5+\y*#1,#1*#2*.5-#1*.5){$\mathstrut\x$};
   \foreach \x [count=\y] in #5
      \node[right,xshift=#1*.5cm+1] at (#1*#3*.5-#1*.5,#1*.5+#1*#2*.5-\y*#1){$\x$};
   \foreach \x [count=\y] in #6
      \node[below,yshift=-#1*.5cm] at (#1*.5+#1*#3*.5-\y*#1,-#1*#2*.5+#1*.5){$\mathstrut\x$};
   \foreach \x [count=\y] in #7
      \node[left,xshift=-#1*.5cm-1] at (-#1*#3*.5+#1*.5,-#1*.5-#1*#2*.5+\y*#1){$\x$};
   \end{tikzpicture}}

\newcommand{\cspace}{.1} 
\newcommand{\tileedge}{.05pt} 
\newcommand{\knotthickness}{1.5pt} 
\newcommand{\graycolor}{gray!40}

\newcommand{\tileo}{\tikz \draw[tilestyle] (0,0) rectangle (\tilesize,\tilesize);}
\newcommand{\tilei}[1][{}]{\tikz {\draw[tilestyle] (0,0) rectangle (\tilesize,\tilesize); \draw[#1,knotstyle={\tilesize*\knotthickness}] (0,.5*\tilesize) to[out=0, in=90] (.5*\tilesize,0);}}
\newcommand{\tileii}[1][{}]{\tikz {\draw[tilestyle] (0,0) rectangle (\tilesize,\tilesize); \draw[#1,knotstyle={\tilesize*\knotthickness}] (\tilesize,.5*\tilesize) to[out=180, in=90] (.5*\tilesize,0);}}
\newcommand{\tileiii}[1][{}]{\tikz {\draw[tilestyle] (0,0) rectangle (\tilesize,\tilesize); \draw[#1,knotstyle={\tilesize*\knotthickness}] (\tilesize,.5*\tilesize) to[out=180, in=-90] (.5*\tilesize,\tilesize);}}
\newcommand{\tileiv}[1][{}]{\tikz {\draw[tilestyle] (0,0) rectangle (\tilesize,\tilesize); \draw[#1,knotstyle={\tilesize*\knotthickness}] (0,.5*\tilesize) to[out=0, in=-90] (.5*\tilesize,\tilesize);}}
\newcommand{\tilev}[1][{}]{\tikz {\draw[tilestyle] (0,0) rectangle (\tilesize,\tilesize); \draw[#1,knotstyle={\tilesize*\knotthickness}] (0,.5*\tilesize) to (\tilesize,.5*\tilesize);}}
\newcommand{\tilevi}[1][{}]{\tikz {\draw[tilestyle] (0,0) rectangle (\tilesize,\tilesize); \draw[#1,knotstyle={\tilesize*\knotthickness}] (.5*\tilesize,0) to (.5*\tilesize,\tilesize);}}
\newcommand{\tilevii}[1][{}]{\tikz {\draw[tilestyle] (0,0) rectangle (\tilesize,\tilesize); \draw[#1,knotstyle={\tilesize*\knotthickness}] (0,.5*\tilesize) to[out=0, in=90] (.5*\tilesize,0); \draw[#1,knotstyle={\tilesize*\knotthickness}] (\tilesize,.5*\tilesize) to[out=180, in=-90] (.5*\tilesize,\tilesize);}}
\newcommand{\tileviii}[1][{}]{\tikz {\draw[tilestyle] (0,0) rectangle (\tilesize,\tilesize); \draw[#1,knotstyle={\tilesize*\knotthickness}] (\tilesize,.5*\tilesize) to[out=180, in=90] (.5*\tilesize,0); \draw[#1,knotstyle={\tilesize*\knotthickness}] (0,.5*\tilesize) to[out=0, in=-90] (.5*\tilesize,\tilesize);}}
\newcommand{\tileix}[1][{}]{\tikz {\draw[tilestyle] (0,0) rectangle (\tilesize,\tilesize); \draw[#1,knotstyle={\tilesize*\knotthickness}] (0,.5*\tilesize) to (\tilesize,.5*\tilesize); \draw[#1,knotstyle={\tilesize*\knotthickness}] (.5*\tilesize,0) to (.5*\tilesize,.5*\tilesize-\tilesize*\cspace); \draw[#1,knotstyle={\tilesize*\knotthickness}] (.5*\tilesize,.5*\tilesize+\tilesize*\cspace) to (.5*\tilesize,\tilesize);}}
\newcommand{\tilex}[1][{}]{\tikz {\draw[tilestyle] (0,0) rectangle (\tilesize,\tilesize); \draw[#1,knotstyle={\tilesize*\knotthickness}] (.5*\tilesize,0) to (.5*\tilesize,\tilesize); \draw[#1,knotstyle={\tilesize*\knotthickness}] (0,.5*\tilesize) to (.5*\tilesize-\tilesize*\cspace,.5*\tilesize); \draw[#1,knotstyle={\tilesize*\knotthickness}] (.5*\tilesize+\tilesize*\cspace,.5*\tilesize) to (\tilesize,.5*\tilesize);}}

\newcommand{\tilec}[1][{}]{\tikz {\draw[tilestyle] (0,0) rectangle (\tilesize,\tilesize); \draw[#1,knotstyle={\tilesize*\knotthickness}] (.5*\tilesize,0) to (.5*\tilesize,\tilesize); \draw[#1,knotstyle={\tilesize*\knotthickness}] (0,.5*\tilesize) to (\tilesize,.5*\tilesize); \draw[#1,knotstyle={.5*\tilesize*\knotthickness}] (.5*\tilesize,.5*\tilesize) circle[radius=.15*\tilesize];}}

\newcommand{\tileuddots}{\tikz{\useasboundingbox (0,0) rectangle (\tilesize,\tilesize); \draw[fill=black] (.3*\tilesize,.3*\tilesize) circle[radius=.5pt];\draw[fill=black] (.5*\tilesize,.5*\tilesize) circle[radius=.5pt];\draw[fill=black] (.7*\tilesize,.7*\tilesize) circle[radius=.5pt];}}
\newcommand{\tiledddots}{\tikz{\useasboundingbox (0,0) rectangle (\tilesize,\tilesize); \draw[fill=black] (.3*\tilesize,.7*\tilesize) circle[radius=.5pt];\draw[fill=black] (.5*\tilesize,.5*\tilesize) circle[radius=.5pt];\draw[fill=black] (.7*\tilesize,.3*\tilesize) circle[radius=.5pt];}}

\tikzset{tilestyle/.style={use as bounding box, line width=\tileedge}} 
\tikzset{knotstyle/.style={line width=#1}} 
\tikzset{mosai/.style={execute at begin cell=\node\bgroup, execute at end cell=\egroup;, column sep=0mm, inner sep=0}}

\newcommand{\tilefit}{.7} 

\title{Virtual Mosaic Knot Theory}
\author{Sandy Ganzell \and Allison Henrich}

\begin{document}

\maketitle

\begin{abstract}
Mosaic diagrams for knots were first introduced in 2008 by Lomanoco and Kauffman for the purpose of building a quantum knot system. Since then, many others have explored the structure of these knot mosaic diagrams, as they are interesting objects of study in their own right. Knot mosaics have been generalized by Gardu\~{n}o to virtual knots, by including an additional tile type to represent virtual crossings. There is another interpretation of virtual knots, however, as knot diagrams on surfaces, which inspires this work. By viewing classical mosaic diagrams as $4n$-gons and gluing edges of these polygons, we obtain knots on surfaces that can be viewed as virtual knots. These virtual mosaics are our present objects of study. In this paper, we provide a set of moves that can be performed on virtual mosaics that preserve knot and link type, we show that any virtual knot or link can be represented as a virtual mosaic, and we provide several computational results related to virtual mosaic numbers for small classical and virtual knots.
\end{abstract}

\section{Introduction}

\subsection{Virtual knot theory}

Introduced by Kauffman in \cite{vkt}, virtual knots can be viewed in at least three different ways: as knot diagrams with an additional crossing type (called a virtual crossing), as Gauss codes, or as knot diagrams on surfaces \cite{cks, abstract, kuperberg}. When viewed as knot diagrams with virtual crossings, a set of virtual Reidemeister moves (or, equivalently, the virtual detour move shown in Figure \ref{detour})
\begin{figure}[htp]
\begin{center}
\includegraphics[width=.8\textwidth]{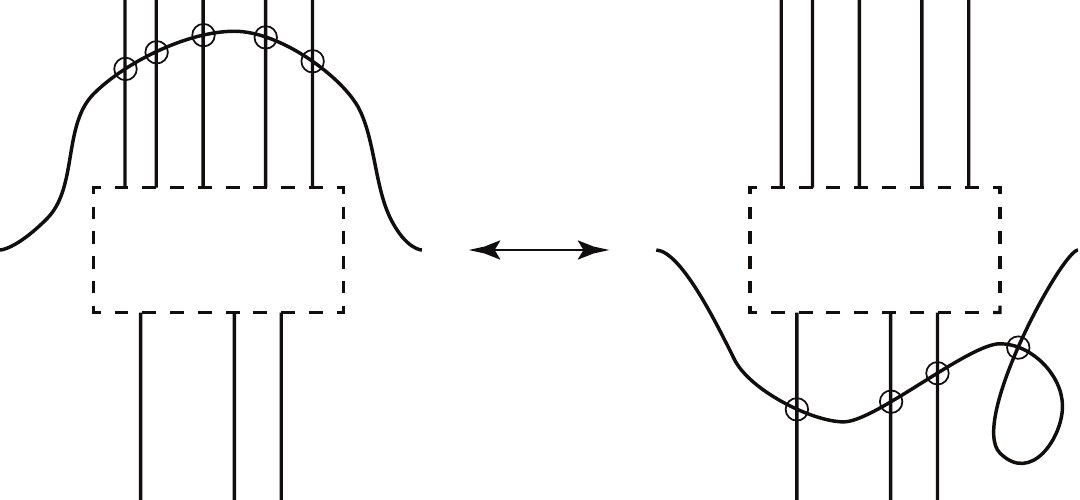}
\end{center}
\caption{The virtual detour move.}\label{detour}
\end{figure}
describes virtual knot equivalence, while Gauss code versions of the ordinary Reidemeister moves define Gauss code equivalence. If virtual knots are viewed as knot diagrams on surfaces, we may perform ordinary Reidemeister moves on these surfaces without changing the virtual knot type, but virtual knot equivalence might also involve changing the surface on which the knot diagram lives, i.e. ``(de)stabilizing." 

\begin{figure}[htp]
\begin{center}
\includegraphics[width=.8\textwidth]{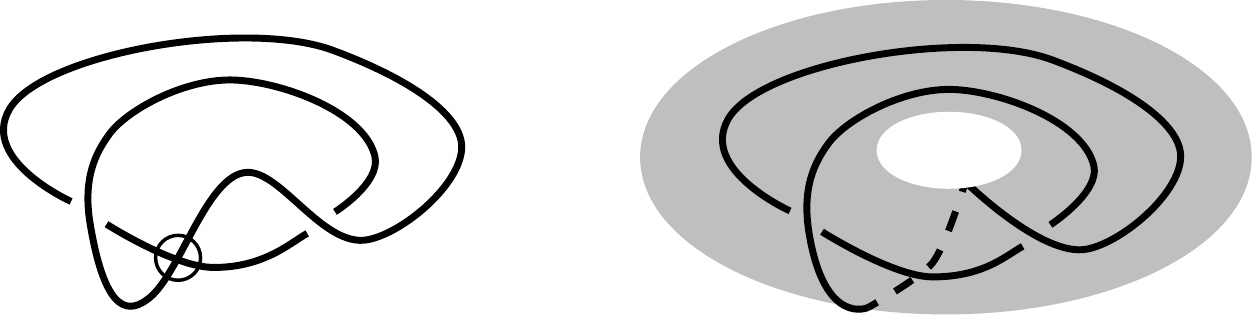}
\end{center}
\caption{The virtual knot, commonly called the \emph{virtual trefoil}, that is given by Gauss code $O1+U2+U1+O2+$.}\label{virtual_ex}
\end{figure}


We will return to virtual knots shortly, but first, we introduce mosaic knots and the objects we aim to study: mosaic representations of virtual knots.

\subsection{Mosaic knots}

In \cite{lomkau}, planar mosaic diagrams were introduced for classical knots as building blocks for developing a quantum knot system. These diagrams are defined to be $n\times n$ grids of \emph{suitably connected} tiles, where each tile is one of the 11 pictured in Figure \ref{tiles}, and suitable connectivity is illustrated in Figure \ref{connect}. 

\begin{figure}[htp]
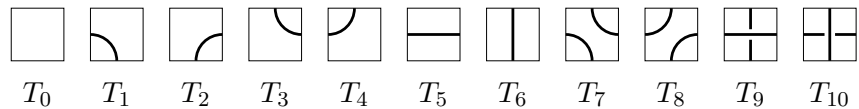

\centering
\setcounter{MaxMatrixCols}{20}$\begin{matrix} 
    \tile[\tilefit]{\tileo} & \tile[\tilefit]{\tilei} & \tile[\tilefit]{\tileii} & \tile[\tilefit]{\tileiii} & \tile[\tilefit]{\tileiv} & \tile[\tilefit]{\tilev} & \tile[\tilefit]{\tilevi} & \tile[\tilefit]{\tilevii} & \tile[\tilefit]{\tileviii} & \tile[\tilefit]{\tileix} & \tile[\tilefit]{\tilex} \\[2.5ex]
    T_0 & T_1 & T_2 & T_3 & T_4 & T_5 & T_6 & T_7 & T_8 & T_9 & T_{10}
\end{matrix}$
\caption{The eleven standard mosaic tiles.}
\label{tiles}
\end{figure}

\begin{figure}[htp]
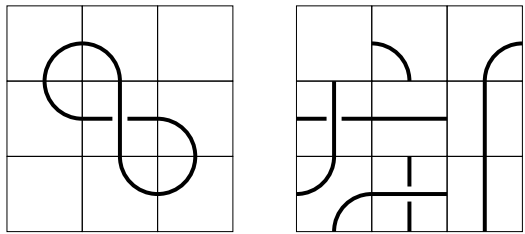

\centering
\begin{tabular}{ccc}
\begin{mosaic}[1]
    \tileii \& \tilei \& \tileo\\ 
    \tileiii \& \tilex \& \tilei\\
    \tileo \& \tileiii \& \tileiv \\
\end{mosaic} &&
\begin{mosaic}[1]
    \tileo \& \tilei \& \tileii\\ 
    \tilex \& \tilev \& \tilevi\\
    \tileviii \& \tileix \& \tilevi \\
\end{mosaic}
\end{tabular}
\caption{A suitably connected knot mosaic (left) and one that is not suitably connected (right).}
\label{connect}
\end{figure}

Many interesting questions related to mosaic knots concern the realizability of knots on mosaics. Lomanoco and Kauffman showed that any knot can be realized on a mosaic \cite{lomkau}, while Kuriya and Shehab proved a more general result: tame knot theory is equivalent to mosaic knot theory \cite{kuriya}. Much of the focus of research on  mosaic knots since then has related to finding the \emph{mosaic number} (i.e., the smallest integer $n$ for which
$K$ is representable as a mosaic knot on an $n\times n$ grid) of specific knots and knot families \cite{lee, involve, ludwig}.

Some variations of classical mosaics have also been introduced and studied. Gardu\~no introduced mosaics for virtual knots, where the collection of mosaic tiles used to create mosaics included an additional virtual crossing tile \cite{garduno}:

\smallskip

\begin{center}\setcounter{MaxMatrixCols}{1}$\begin{matrix} 
    \tile[\tilefit]{\tilec} 
\end{matrix}$\end{center}

\smallskip

\noindent Building on Gardu\~no's work, results in \cite{dye} give bounds relating mosaic number and crossing number for these virtual mosaic knots. 

Another variation on mosaic knots was introduced by Carlisle and Laufer, who studied \emph{toroidal mosaic knots} in \cite{torus}. Toroidal mosaic knots are defined by identifying opposite edges of the $n\times n$ mosaic grid.


\subsection{Virtual Mosaics}

In this work, we represent virtual knots and links not via their virtual diagrams, but as knot diagrams on orientable surfaces.

\begin{defn}
A \emph{virtual $n$-mosaic} or \emph{virtual mosaic} (if $n$ is unspecified) is an $n\times n$ array of standard mosaic tiles, together with an identification of the $4n$ edges of the array boundary, that forms a knot or link diagram on a closed, orientable surface. We say the (classical or virtual) knot or link $L$ is \emph{represented} by the virtual mosaic, and the genus of the surface will also be called the genus of the virtual mosaic.
\end{defn}

As a first example, consider the virtual 2-mosaics pictured in Figure \ref{diff_genus_vtref}. It is easy to obtain a Gauss code for the knots represented by each mosaic. We then see that the two mosaics represent the same knot, namely the virtual trefoil, shown in Figure \ref{virtual_ex}. Note that the mosaic on the left has genus 1 while the mosaic on the right has genus 2. 

\begin{figure}[htp]
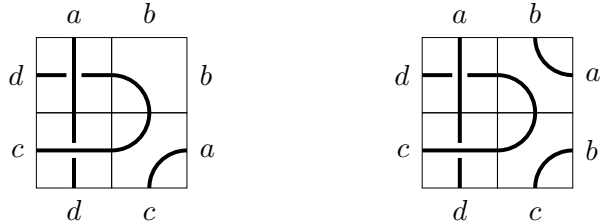

\centering
\begin{vmosaic}[1]{2}{2}{{a,b}}{{b,a}}{{c,d}}{{c,d}}
\tilex \& \tilei \\
\tileix \& \tileviii \\
\end{vmosaic}\hspace{2cm}
\begin{vmosaic}[1]{2}{2}{{a,b}}{{a,b}}{{c,d}}{{c,d}}
\tilex \& \tilevii \\
\tileix \& \tileviii \\
\end{vmosaic}
\caption{Two representations of the virtual trefoil.}
\label{diff_genus_vtref}
\end{figure}

We can obtain a virtual knot or link diagram associated to a certain virtual mosaic by drawing arcs connecting endpoints that lie on boundary components sharing the same label. If any crossings occur outside of the mosaic between connecting arcs, these crossings are said to be virtual. See Figure \ref{virt-knot} for an example. Note that the virtual knot or link type of the result is independent of how we draw connecting arcs since all possible arcs drawn in this way are related by the virtual detour move. 

\begin{figure}[htp]
\centering
\begin{vmosaic}[1]{2}{2}{{a,b}}{{b,a}}{{c,d}}{{c,d}}
\tilex \& \tilei \\
\tileix \& \tileviii \\
\end{vmosaic}\meq[1em]{\longrightarrow}\hspace{-2.3cm}
\begin{tikzpicture}[baseline]
\begin{mosaic}
\tilex \& \tilei \\
\tileix \& \tileviii \\
\end{mosaic}
\draw[line width=\knotthickness] (-1.5,1) to[out=90, in=135, looseness=.8] (.25,1.25) to[out=-45, in=0, looseness=.8] (0,-.5);
\draw[line width=\knotthickness] (-.5,-1) to[out=-90, in=-90] (-2.25,-1.25) to[out=90, in=180] (-2,-.5);
\draw[line width=\knotthickness] (-1.5,-1) to[out=-90, in=0] (-2.25,-1.25) to[out=180, in=180] (-2,.5);
\draw[thick] (-2.25,-1.25) circle (1.4mm);
\end{tikzpicture}
\caption{A virtual knot diagram obtained from the virtual mosaic pictured in Figure \ref{diff_genus_vtref} (L).}
\label{virt-knot}
\end{figure}
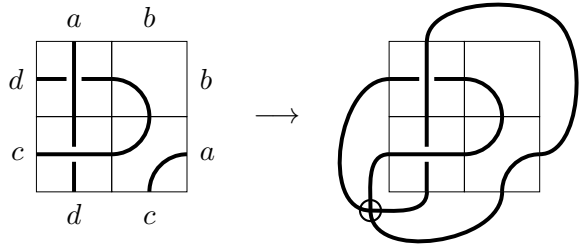

One key virtual knot invariant we wish to study in this paper is called the virtual mosaic number, defined as follows.

\begin{defn}
The \emph{virtual mosaic number} of $L$, denoted $\vmn{L}$, is the smallest integer $n$ for which $L$ can be represented by a virtual $n$-mosaic.
\end{defn}

In Section \ref{sec:ex}, we determine virtual mosaic numbers for small-crossing classical and virtual knots. Since we aim to represent classical knots on genus 0 surfaces, it will be helpful for us to recall the following standard result.

\begin{prop}
The genus 0 virtual mosaics correspond to identifications in which all label pairs are nested, i.e., for all edge labels $x$ and $y$, between two $x$-labels there are either zero or two $y$-labels.
\end{prop}

We now turn our focus to studying examples of virtual mosaics.

\section{Examples}\label{sec:ex}

\subsection{The simplest classical knots and links}

We begin by observing that the unknot and 2-component unlink both have virtual mosaic number 1. Both can be drawn as $1\times 1$ virtual mosaics with genus 0. See Figure \ref{unknot-unlink}.

\begin{figure}[htp]
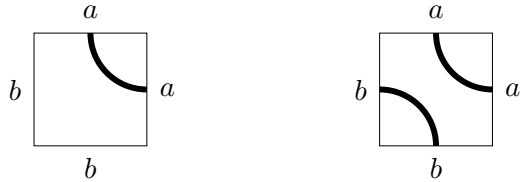

\centering
\begin{vmosaic}[1.5]{1}{1}{{a}}{{a}}{{b}}{{b}} 
\tileiii \\
\end{vmosaic}\hspace{2cm}
\begin{vmosaic}[1.5]{1}{1}{{a}}{{a}}{{b}}{{b}} 
\tilevii \\
\end{vmosaic}
\caption{Unknot and 2-component unlink, both genus 0.}
\label{unknot-unlink}
\end{figure}

On a $1\times 1$ virtual mosaic, there can be at most one crossing and at most two components. Thus, only one other link could have virtual mosaic number 1, namely the virtual Hopf link (Figure \ref{virtualhopf}). 

Which knots have virtual mosaic number 2? Since there are only four tiles available for crossings, we need only consider knots that have crossing number four or less. As we will soon see, all 2- and 3-crossing classical and virtual knots have virtual mosaic number 2. But while many 4-crossing virtual knots can be represented as virtual 2-mosaics, the classical figure-8 knot cannot. To show this, we begin with a lemma regarding the Gauss code of the figure-8 knot, knot $4_1$.

\begin{lem}\label{gauss-fig-8}
A Gauss code for the classical figure-8 knot with four crossings must contain a sequence of four distinct consecutive crossings. 
\end{lem}
\begin{proof}
Without loss of generality, suppose the sequence begins \verb`12`. (Otherwise an $R1$ move would remove the crossing.) If the next crossing is \verb`1` then \verb`1` would be an odd crossing; if the next crossing is \verb`2`, there would be an $R1$ move. Thus the sequence begins \verb`123`. If the next crossing is \verb`4` then we have our desired sequence. Otherwise, then next crossing must be \verb`1` (to avoid odd crossings and $R1$ moves). Following \verb`1231` must be \verb`4`, providing the desired sequence. (A sequence beginning \verb`12312` must result in an odd crossing or $R1$ move.)
\end{proof}

\begin{prop}
The figure-8 knot has virtual mosaic number 3.
\end{prop}
\begin{proof}
We show that the figure-8 cannot be drawn on a $2\times 2$ mosaic with any genus. Since the figure-8 knot is alternating and has 4 classical crossings, its unlabeled mosaic must have the form shown in Figure \ref{fig8}(i).

\begin{figure}[htp]
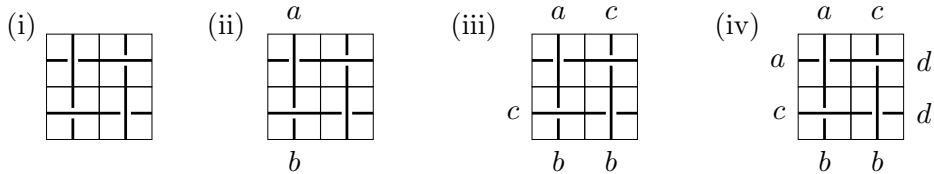

\centering
\raisebox{7mm}{(i) }\begin{mosaic}[.7]
\tilex \& \tileix \\
\tileix \& \tilex \\
\end{mosaic}\hspace{6mm}
\raisebox{7mm}{(ii)}\begin{vmosaic}[.7]{2}{2}{{a,}}{{,}}{{,b}}{{,}}
\tilex \& \tileix \\
\tileix \& \tilex \\
\end{vmosaic}\hspace{6mm}
\raisebox{7mm}{(iii)}\begin{vmosaic}[.7]{2}{2}{{a,c}}{{,}}{{b,b}}{{c,}}
\tilex \& \tileix \\
\tileix \& \tilex \\
\end{vmosaic}\hspace{6mm}
\raisebox{7mm}{(iv)}\begin{vmosaic}[.7]{2}{2}{{a,c}}{{d,d}}{{b,b}}{{c,a}}
\tilex \& \tileix \\
\tileix \& \tilex \\
\end{vmosaic}
\caption{Labeling an alternating $2\times 2$ virtual mosaic.}
\label{fig8}
\end{figure}

By Lemma \ref{gauss-fig-8}, a Gauss code for the figure-8 with four crossings must contain an alternating sequence \verb`1234`. Without loss of generality, suppose this sequence begins at the left north edge. If we label that edge $a$, then the left south edge must have a different label (otherwise we would have a link). See Figure \ref{fig8}(ii).

To ensure the knot passes through the sequence \verb`1234`, the other $b$-label must be the right south edge. (The right north edge would create a nonalternating 4-crossing knot.) Label the right north edge $c$. To keep the knot alternating, the other $c$-label must be the lower west edge, as in Figure \ref{fig8}(iii). (Note that the upper east edge would create an $R1$ move, resulting in a knot with fewer than four crossings.)

Continuing with alternating crossings, the remaining labels must produce the virtual mosaic shown in Figure \ref{fig8}(iv), which is the trefoil (with an additional $R1$ move at the adjacent $a$ edges), not the figure-8 knot.

As we can see in Figure \ref{trefoilfigure8}, the figure-8 knot, $4_1$, can be drawn on a $3\times 3$ virtual mosaic with genus 0. Hence, $m_v(4_1)=3$.
\end{proof}

\begin{figure}[htp]
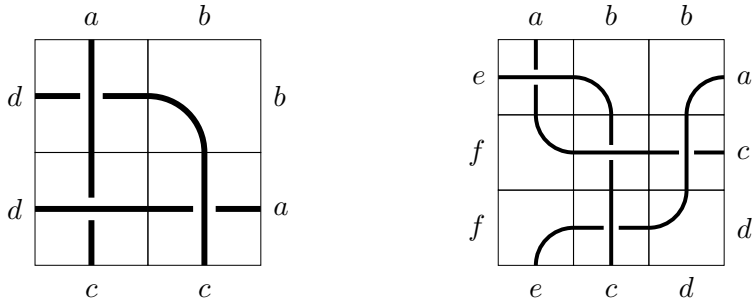

\centering
\begin{vmosaic}[1.5]{2}{2}{{a,b}}{{b,a}}{{c,c}}{{d,d}} 
\tilex \& \tilei \\
\tileix \& \tilex \\
\end{vmosaic}\hspace{2cm}
\begin{vmosaic}[1]{3}{3}{{a,b,b}}{{a,c,d}}{{d,c,e}}{{f,f,e}} 
\tileix \& \tilei \& \tileii \\
\tileiii \& \tileix \& \tilex \\
\tileii \& \tilex \& \tileiv \\
\end{vmosaic}
\caption{The classical trefoil and figure-8 knots, both genus 0.}
\label{trefoilfigure8}
\end{figure}

In Figure \ref{trefoilfigure8}, along with a representation of the figure-8 knot on a virtual 3-mosaic, we see an example demonstrating that the virtual mosaic number of the trefoil is 2. Since we've determined the virtual mosaic numbers for all (three) classical knots with four or fewer crossings, let's turn our attention to classical knots with five or more crossings.

\subsection{Classical knots with five or more crossings}

Classical knots with 5, 6 and 7 crossings can be realized on the smallest mosaics necessary to contain their crossing tiles. 

\begin{prop}
All classical 5- 6- and 7-crossing knots have virtual mosaic number 3.
\end{prop}

\begin{proof} Classical knots with 5 or more crossings cannot fit on a $2\times 2$ mosaic of any genus, but Figure \ref{knots-51-52} shows that both 5-crossing classical knots can fit on a $3\times 3$ mosaic. Similarly, Figure \ref{knots-61-62-63} provides virtual 3-mosaics for knots $6_1$, $6_2$ and $6_3$. All seven crossing knots are illustrated as genus 0 virtual 3-mosaics in Appendix \ref{minimal-mosaics}.\end{proof}

\begin{figure}[htp]
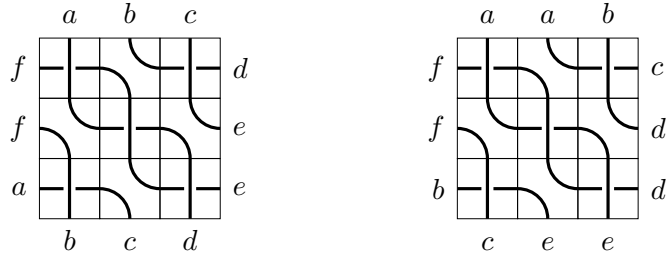

\centering
\begin{vmosaic}[.8]{3}{3}{{a,b,c}}{{d,e,e}}{{d,c,b}}{{a,f,f}} 
\tilex \& \tilevii \& \tilex \\
\tilevii \& \tilex \& \tilevii \\
\tilex \& \tilevii \& \tilex \\
\end{vmosaic}\hspace{2cm}
\begin{vmosaic}[.8]{3}{3}{{a,a,b}}{{c,d,d}}{{e,e,c}}{{b,f,f}} 
\tilex \& \tilevii \& \tilex \\
\tilevii \& \tilex \& \tilevii \\
\tilex \& \tilevii \& \tilex \\
\end{vmosaic}
\caption{Classical knots $5_1$ and $5_2$, both genus 0.}
\label{knots-51-52}
\end{figure}

\begin{figure}[htp]
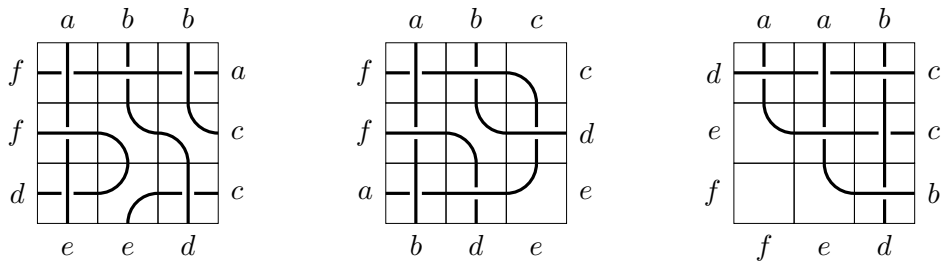

\centering
\begin{vmosaic}[.8]{3}{3}{{a,b,b}}{{a,c,c}}{{d,e,e}}{{d,f,f}} 
\tilex \& \tileix \& \tilex \\
\tileix \& \tilevii \& \tilevii \\
\tilex \& \tileviii \& \tilex \\
\end{vmosaic}\hfill
\begin{vmosaic}[.8]{3}{3}{{a,b,c}}{{c,d,e}}{{e,d,b}}{{a,f,f}} 
\tilex \& \tileix \& \tilei \\
\tileix \& \tilevii \& \tileix \\
\tilex \& \tileix \& \tileiv \\
\end{vmosaic}\hfill
\begin{vmosaic}[.8]{3}{3}{{a,a,b}}{{c,c,b}}{{d,e,f}}{{f,e,d}} 
\tileix \& \tilex \& \tileix \\
\tileiii \& \tileix \& \tilex \\
\tileo \& \tileiii \& \tileix \\
\end{vmosaic}
\caption{Knots $6_1$, $6_2$, $6_3$, all genus 0.}
\label{knots-61-62-63}
\end{figure}

\begin{example} It is possible for the virtual mosaic number of a classical knot to be realized only in a genus 0 virtual mosaic with more crossings than the crossing number of the knot. Knot $7_1$ is an example. By an exhaustive analysis of all genus 0 edge identifications for all 7-crossing configurations on a $3\times 3$ grid, we observed that there is no virtual 3-mosaic with 7 crossings that represents knot $7_1$. Yet, the virtual mosaic number of $7_1$ is 3, since there is a 9-crossing virtual 3-mosaic that represents the knot, pictured in Figure \ref{71-4-mosaic}.\end{example}

\begin{figure}[htp]
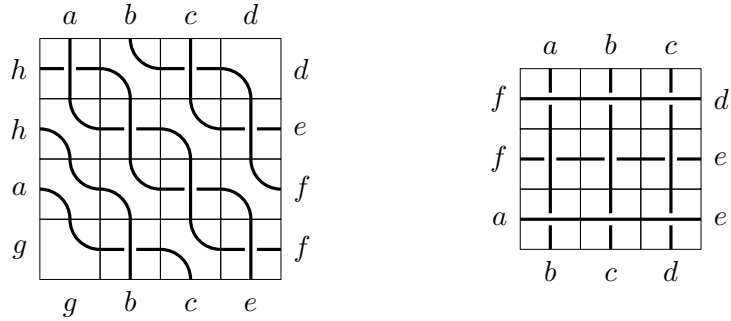

\centering
\begin{vmosaic}[.8]{4}{4}{{a,b,c,d}}{{d,e,f,f}}{{e,c,b,g}}{{g,a,h,h}} 
\tilex \& \tilevii \& \tilex \& \tilei \\
\tilevii \& \tilex \& \tilevii \& \tilex \\
\tilevii \& \tilevii \& \tilex \& \tilevii \\
\tileiii \& \tilex \& \tilevii \& \tilex \\
\end{vmosaic}\hspace{2cm}
\begin{vmosaic}[.8]{3}{3}{{a,b,c}}{{d,e,e}}{{d,c,b}}{{a,f,f}} 
\tileix \& \tileix \& \tileix \\
\tilex \& \tilex \& \tilex \\
\tileix \& \tileix \& \tileix \\
\end{vmosaic}
\caption{Knot $7_1$ as a 7-crossing 4-mosaic (left) and as a 9-crossing 3-mosaic (right), both with genus 0.}
\label{71-4-mosaic}
\end{figure}

A computer search reveals that the only 8-crossing classical knots with virtual mosaic number 3 are: $8_5$, $8_7$, $8_8$, $8_{10}$, $8_{12}$, $8_{13}$, $ 8_{14}$, $8_{15}$, $8_{19}$, $8_{20}$, and $8_{21}$. The remaining 8-crossing classical knots have virtual mosaic number 4. See Appendix \ref{minimal-mosaics}. 

Furthermore, of the 9-crossing knots, only $9_{16}$, $9_{23}$, and $9_{31}$ have virtual mosaic number 3. See Figure \ref{9cross}. Notice that all three of these 9-crossing knots are alternating. Interestingly, no genus 0 virtual mosaics with non-alternating crossing patterns on a $3\times 3$ grid represent 9-crossing knots. We suspect, but have not proven, that the remaining 9-crossing classical knots all have virtual mosaic number 4. 

\begin{figure}[htp]
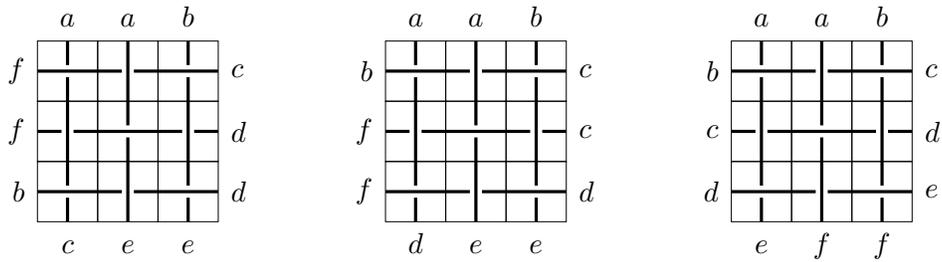

\centering
\begin{vmosaic}[.8]{3}{3}{{a,a,b}}{{c,d,d}}{{e,e,c}}{{b,f,f}} 
\tileix \& \tilex \& \tileix \\
\tilex \& \tileix \& \tilex \\
\tileix \& \tilex \& \tileix \\
\end{vmosaic}\hfill%
\begin{vmosaic}[.8]{3}{3}{{a,a,b}}{{c,c,d}}{{e,e,d}}{{f,f,b}} 
\tileix \& \tilex \& \tileix \\
\tilex \& \tileix \& \tilex \\
\tileix \& \tilex \& \tileix \\
\end{vmosaic}\hfill%
\begin{vmosaic}[.8]{3}{3}{{a,a,b}}{{c,d,e}}{{f,f,e}}{{d,c,b}} 
\tileix \& \tilex \& \tileix \\
\tilex \& \tileix \& \tilex \\
\tileix \& \tilex \& \tileix \\
\end{vmosaic}
\caption{Knots $9_{16}$, $9_{23}$, $9_{31}$.}
\label{9cross}
\end{figure}

\subsection{Small crossing virtual knots and links}

 In addition to determining virtual mosaic numbers for classical knots, we may determine virtual mosaic numbers for many of the virtual knots on Green's virtual knot table \cite{green}. For instance, all virtual 2-crossing and 3-crossing knots have virtual mosaic number 2, as illustrated in Figures \ref{virtualhopf}, \ref{virtual3_1-etc}, and \ref{virtual3_4-etc}. (Note that virtual knot 3.6 is the classical trefoil.)

\begin{figure}[htp]
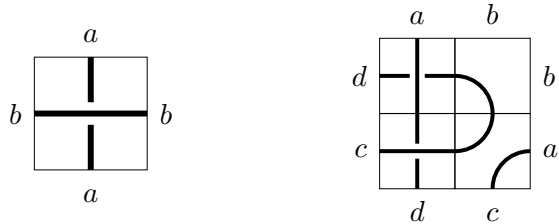

\centering
\begin{vmosaic}[1.5]{1}{1}{{a}}{{b}}{{a}}{{b}} 
\tileix \\
\end{vmosaic}\hspace{2cm}
\begin{vmosaic}[1]{2}{2}{{a,b}}{{b,a}}{{c,d}}{{c,d}} 
\tilex \& \tilei \\
\tileix \& \tileviii \\
\end{vmosaic}
\caption{Virtual Hopf link and virtual trefoil, both genus 1.}
\label{virtualhopf}
\end{figure}
\begin{figure}[htp]
\centering
\begin{vmosaic}[1]{2}{2}{{a,b}}{{c,d}}{{a,b}}{{c,d}}
\tilex \& \tileviii \\
\tilex \& \tileix \\
\end{vmosaic}\hspace{1cm}
\begin{vmosaic}[1]{2}{2}{{a,b}}{{c,a}}{{d,b}}{{c,d}}
\tilex \& \tileviii \\
\tilex \& \tilex \\
\end{vmosaic}\hspace{1cm}
\begin{vmosaic}[1]{2}{2}{{a,b}}{{c,d}}{{a,b}}{{c,d}}
\tilex \& \tileviii \\
\tilex \& \tilex \\
\end{vmosaic}
\caption{Virtual knots 3.1, 3.2, 3.3; genera 2, 1, 2, respectively.}
\label{virtual3_1-etc}
\end{figure}
\begin{figure}[htp]
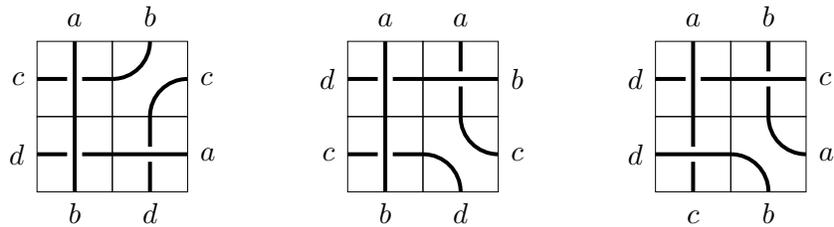

\centering
\begin{vmosaic}[1]{2}{2}{{a,b}}{{c,a}}{{d,b}}{{d,c}}
\tilex \& \tileviii \\
\tilex \& \tileix \\
\end{vmosaic}\hspace{1cm}
\begin{vmosaic}[1]{2}{2}{{a,a}}{{b,c}}{{d,b}}{{c,d}}
\tilex \& \tileix \\
\tilex \& \tilevii \\
\end{vmosaic}\hspace{1cm}
\begin{vmosaic}[1]{2}{2}{{a,b}}{{c,a}}{{b,c}}{{d,d}}
\tilex \& \tileix \\
\tileix \& \tilevii \\
\end{vmosaic}
\caption{Virtual knots 3.4, 3.5, 3.7; genera 2, 1, 1, respectively.}
\label{virtual3_4-etc}
\end{figure}

Many 4-crossing virtual knots also have virtual mosaic number 2. We conducted an exhaustive search for virtual knots in all virtual 2-mosaics using Miller's Virtual KnotFolio \cite{vknotfolio} Green's virtual knot table \cite{green} (identifying virtual knots using the 2- and  3-cabled Jones polynomials). This search demonstrated that the following 4-crossing virtual knots have virtual mosaic number 2: 4.1, 4.4, 4.8, 4.12, 4.14, 4.21, 4.30, 4.36, 4.37, 4.43, 4.48, 4.55, 4.59, 4.64, 4.65, 4.71, 4.77, 4.92, 4.95, 4.99, 4.104, 4.105. See Appendix \ref{minimal-virtual-mosaics} for virtual mosaic diagrams of all virtual knots $K$ with $\vmn K=2$.

All of the remaining 4-crossing virtual knots have virtual mosaic number at least 3. We suspect, although we have not proven, that these 4-crossing virtual knots have virtual mosaic number exactly equal to 3.

\section{Relation to Classical Mosaics}

Let $\mn{L}$ denote the (classical) mosaic number of the classical knot or link $L$. We would like to know what the relationship is between $\mn{L}$ and $\vmn{L}$. We have the following result.

\begin{prop}\label{class-virt-rel}
If $L$ is a link or a nontrivial knot, then $\vmn{L}\le\mn{L}-2$.
\end{prop}
\begin{proof}
Since $L$ is not the unknot, we have that $\mn{L}\ge4$. (Otherwise the mosaic for $L$ would contain at most one crossing tile.) We now form a virtual $(n-2)$-mosaic for $L$. Note that the tiles in the first and last rows and columns of the mosaic cannot be $T_9$ or $T_{10}$ (the crossing tiles). Delete these rows and columns and identify edges to establish the original connections. Label any remaining edges in matched pairs. 
\end{proof}

\begin{example}\label{ex:7_2}
Consider knot $7_2$. It is shown in \cite{involve} that $\mn{7_2}=6$. The construction of a virtual 4-mosaic representing $7_2$ is shown in Figure \ref{virtual-from-classical}.
\begin{figure}[htp]
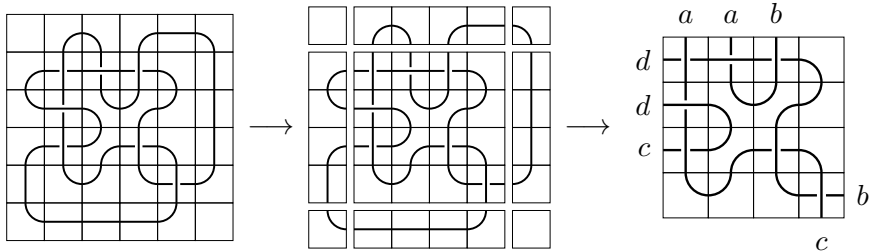

    \centering
    \begin{mosaic}[.5] 
    \tileo \& \tileii \& \tilei \& \tileii \& \tilev \& \tilei \\
    \tileii \& \tilex \& \tileix \& \tilex \& \tilei \& \tilevi \\
    \tileiii \& \tileix \& \tilevii \& \tileviii \& \tileiv \& \tilevi \\
    \tileii \& \tilex \& \tileviii \& \tilex \& \tilei \& \tilevi \\
    \tilevi \& \tileiii \& \tileiv \& \tileiii \& \tilex \& \tileiv \\
    \tileiii \& \tilev \& \tilev \& \tilev \& \tileiv \& \tileo \\
    \end{mosaic}\meq[.5em]{\longrightarrow}%
    \begin{mosaic}[.5]
    \tileo\hspace{1mm} \& \tileii \& \tilei \& \tileii \& \tilev \& \hspace{1mm}\tilei \\[1mm]
    \tileii\hspace{1mm} \& \tilex \& \tileix \& \tilex \& \tilei \& \hspace{1mm}\tilevi \\
    \tileiii\hspace{1mm} \& \tileix \& \tilevii \& \tileviii \& \tileiv \& \hspace{1mm}\tilevi \\
    \tileii\hspace{1mm} \& \tilex \& \tileviii \& \tilex \& \tilei \& \hspace{1mm}\tilevi \\
    \tilevi\hspace{1mm} \& \tileiii \& \tileiv \& \tileiii \& \tilex \& \hspace{1mm}\tileiv \\[1mm]
    \tileiii\hspace{1mm} \& \tilev \& \tilev \& \tilev \& \tileiv \& \hspace{1mm}\tileo \\
    \end{mosaic}\meq[.5em]{\longrightarrow}%
    \begin{vmosaic}[.6]{4}{4}{{a,a,b,}}{{,,,b}}{{c,,,}}{{,c,d,d}}
    \tilex \& \tileix \& \tilex \& \tilei \\
    \tileix \& \tilevii \& \tileviii \& \tileiv \\
    \tilex \& \tileviii \& \tilex \& \tilei \\
    \tileiii \& \tileiv \& \tileiii \& \tilex \\
    \end{vmosaic}
    \caption{Constructing a virtual 4-mosaic from a classical 6-mosaic.}
    \label{virtual-from-classical}
\end{figure}
Note that the resulting virtual 4-mosaic is not minimal. Figure \ref{knot72} shows that $\vmn{7_2}=3$.
\begin{figure}[htp]
    \centering
    \begin{vmosaic}[.8]{3}{3}{{a,b,b}}{{c,d,d}}{{c,e,e}}{{a,f,f}} 
    \tilex \& \tileix \& \tilex \\
    \tileix \& \tileviii \& \tilevii \\
    \tilex \& \tileix \& \tilex \\
    \end{vmosaic}
    \caption{Knot $7_2$ has virtual mosaic number 3.}
    \label{knot72}
\end{figure}

Thus, we see that the bound given in Proposition \ref{class-virt-rel} is not sharp.
\end{example}

\section{Virtual Mosaic Moves}

In this section, we focus our attention on virtual mosaic equivalence. We provide a collection of virtual mosaic moves that do not affect knot type. 

In the diagrams that follow, mosaic arcs in light gray are optional and may or may not be present. Arcs that are dotted may replace overlapping black arcs. For instance, \raisebox{.125cm}{\begin{mosaic}[.5] \rlap{\tilevi[densely dotted]}\tileiv \\ \end{mosaic}} is taken to mean \raisebox{.125cm}{\begin{mosaic}[.5] \tileiv \\ \end{mosaic}} or \raisebox{.125cm}{\begin{mosaic}[.5] \tilevi \\ \end{mosaic}}.

In each equivalence that follows, reflections and rotations of the moves illustrated are also allowed. In addition, we omit any move that is an exact replica of one pictured, except with all \raisebox{.125cm}{\begin{mosaic}[.5] \tileix \\ \end{mosaic}} ($T_9$) tiles replaced with \raisebox{.125cm}{\begin{mosaic}[.5] \tilex \\ \end{mosaic}} ($T_{10}$) tiles and vice versa. Unlabeled edges may be at the edge or in the middle of the mosaic; labeled edges must be on the edge. 

\subsection{Classical mosaic moves (KL moves)}\label{KL}

In the interior of a virtual mosaic grid, classical mosaic moves, introduced by Lomonaco and Kauffman in \cite{lomkau}, may be applied without changing the virtual mosaic's knot or link type. We refer to these moves as KL moves. Names of moves are taken from \cite{lomkau}, with the convention that moves that differ by one of the symmetries described above are listed once and both names are provided.

\subsubsection{Planar isotopy moves}
\begin{align*}
    &\begin{mosaic}[.7]\rlap{\tileiv[\graycolor]}\tileii \&  \tileiv\\ \tileiv \& \tileii[\graycolor]\\ \end{mosaic}\meq[.5em]{\xleftrightarrow[\rule{9mm}{0pt}]{P_1}}\begin{mosaic}[.7]\tileiv[\graycolor] \& \tilevi \\ \tilev \& \rlap{\tileii[\graycolor]}\tileiv \\ \end{mosaic}
    \hspace{1cm} \begin{mosaic}[.7] \tilev \&  \rlap{\tileiii[\graycolor]}\tilei\\ \tileii \& \rlap{\tileii[\graycolor]}\tileiv\\ \end{mosaic}\meq[.5em]{\xleftrightarrow[\rule{9mm}{0pt}]{P_2/P_3}}\begin{mosaic}[.7]\tilei \& \tileiii[\graycolor] \\ \tilevi \& \tileii[\graycolor] \\ \end{mosaic}
\\[3mm]
    &\begin{mosaic}[.7]\tileiv[\graycolor] \&  \tileii\\ \tilev \& \rlap{\tileii[\graycolor]}\tileiv\\ \end{mosaic}\meq[.5em]{\xleftrightarrow[\rule{9mm}{0pt}]{P_4}}\begin{mosaic}[.7]\rlap{\tileiv[\graycolor]}\tileii \& \tilev \\ \tileiv \& \tileii[\graycolor] \\ \end{mosaic}
    \hspace{1cm} \begin{mosaic}[.7] \tileiv[\graycolor] \&  \tileiii[\graycolor]\\ \tilev \& \tilev\\ \end{mosaic}\meq[.5em]{\xleftrightarrow[\rule{9mm}{0pt}]{P_5}}\begin{mosaic}[.7]\rlap{\tileiv[\graycolor]}\tileii \& \rlap{\tileiii[\graycolor]}\tilei \\ \tileiv \& \tileiii \\ \end{mosaic}
\\[3mm]
    &\begin{mosaic}[.7]\tilei \&  \tileiii[\graycolor]\\ \tileiv \& \tileii[\graycolor]\\ \end{mosaic}\meq[.5em]{\xleftrightarrow[\rule{9mm}{0pt}]{P_6}}\begin{mosaic}[.7]\tilev \& \rlap{\tileiii[\graycolor]}\tilei \\ \tilev \& \rlap{\tileii[\graycolor]}\tileiv \\ \end{mosaic}
    \hspace{1cm} \begin{mosaic}[.7] \tileiv[\graycolor] \&  \tileiii[\graycolor]\\ \tilei \& \tileii[\graycolor] \\ \end{mosaic}\meq[.5em]{\xleftrightarrow[\rule{9mm}{0pt}]{P_7}}\begin{mosaic}[.7]\rlap{\tileiv[\graycolor]}\tileii \& \rlap{\tileiii[\graycolor]}\tilei \\ \tileviii \& \rlap{\tileii[\graycolor]}\tileiv \\ \end{mosaic}
\\[3mm]
    &\begin{mosaic}[.7]\rlap{\tileiv[\graycolor]}\tileii \&  \tileviii\\ \tileix \& \rlap{\tileii[\graycolor]}\tileiv\\ \end{mosaic}\meq[.5em]{\xleftrightarrow[\rule{9mm}{0pt}]{P_8/P_9}}\begin{mosaic}[.7]\rlap{\tileiv[\graycolor]}\tileii \&  \tileix\\ \tileviii \& \rlap{\tileii[\graycolor]}\tileiv\\ \end{mosaic}
    \hspace{1cm} \begin{mosaic}[.7] \tilevii \&  \tileix\\ \rlap{\tilei[\graycolor]}\tileiii \& \rlap{\tileii[\graycolor]}\tileiv \\ \end{mosaic}\meq[.5em]{\xleftrightarrow[\rule{9mm}{0pt}]{\mathclap{P_{10}/P_{11}}}}\begin{mosaic}[.7] \tilex \&  \tileviii\\ \rlap{\tilei[\graycolor]}\tileiii \& \rlap{\tileii[\graycolor]}\tileiv \\ \end{mosaic}
\end{align*}
\subsubsection{Reidemeister moves}
\begin{align*}
    &\begin{mosaic}[.7]\rlap{\tileiv[\graycolor]}\tileii \&  \rlap{\tileiii[\graycolor]}\tilei\\ \tileix \& \rlap{\tileii[\graycolor]}\tileiv\\ \end{mosaic}\meq[.5em]{\xleftrightarrow[\rule{9mm}{0pt}]{R_1/R'_1}}\begin{mosaic}[.7]\rlap{\tileiv[\graycolor]}\tileii \&  \rlap{\tileiii[\graycolor]}\tilei\\ \tileviii \& \rlap{\tileii[\graycolor]}\tileiv\\ \end{mosaic}
\\[3mm]
    &\begin{mosaic}[.7]\rlap{\tileiv[\graycolor]}\tileii \&  \tileviii\\ \tileviii \& \rlap{\tileii[\graycolor]}\tileiv\\ \end{mosaic}\meq[.5em]{\xleftrightarrow[\rule{9mm}{0pt}]{R_2/R'_2}}\begin{mosaic}[.7]\rlap{\tileiv[\graycolor]}\tileii \&  \tilex\\ \tileix \& \rlap{\tileii[\graycolor]}\tileiv\\ \end{mosaic}
    \hspace{1cm} \begin{mosaic}[.7] \rlap{\tileiv[\graycolor]}\tileii \&  \rlap{\tileiii[\graycolor]}\tilei\\ \tileix \& \tileix \\ \end{mosaic}\meq[.5em]{\xleftrightarrow[\rule{9mm}{0pt}]{R''_2/R'''_2}}\begin{mosaic}[.7] \rlap{\tileiv[\graycolor]}\tileii \&  \rlap{\tileiii[\graycolor]}\tilei\\ \tileviii \& \tilevii \\ \end{mosaic}
\\[3mm]
&\begin{mosaic}[.7]
\rlap{\tilei[densely dotted]}\tilevi \& \tilevi \& \tileiii[\graycolor] \\
\tileix \& \tileix \& \tilev \\
\rlap{\tilei[\graycolor]}\tileiii \& \tilex \& \rlap{\tilei[densely dotted]}\tilev \\
\end{mosaic}
\meq[.5em]{\xleftrightarrow[\rule{9mm}{0pt}]{R_3^{(n)}}}
\begin{mosaic}[.7]
\rlap{\tilev[densely dotted]}\tileiii \& \tilex \& \rlap{\tileiii[\graycolor]}\tilei \\
\tilev \& \tileix \& \tileix \\
\tilei[\graycolor] \& \tilevi \& \rlap{\tilevi[densely dotted]}\tileiii \\
\end{mosaic}
\end{align*}

\subsection{Surface isotopies}\label{surf-iso}

The following are additional isotopies needed to capture the isotopies of a virtual knot or link that involve the boundary of the mosaic.

\begin{align*}
&\begin{vmosaic}[.8]{1}{2}{{x,x}}{{}}{{}}{{}} \rlap{\tilevi[densely dotted]}\tileiv \& \rlap{\tilevi[dashed]}\tileiii \\ \end{vmosaic}
\meq[.5em]{\xleftrightarrow[\rule{9mm}{0pt}]{\mathit{SI}_{\!1}}}
\begin{vmosaic}[.8]{1}{2}{{x,x}}{{}}{{}}{{}} \rlap{\tileii[densely dotted]}\tilev \& \rlap{\tilei[dashed]}\tilev \\ \end{vmosaic}
\\[3mm]
&\raisebox{1.4cm}{\begin{vmosaic}[.8]{1}{3}{{x,y,z}}{{}}{{}}{{}} \rlap{\tileii[densely dotted]}\tilev \& \tilex \& \rlap{\tilei[dashed]}\tilev \\ \end{vmosaic}}\hspace{2mm}
\begin{vmosaic}[.8]{3}{1}{{}}{{x,y,z}}{{}}{{}} \rlap{\tileiv[\graycolor]}\tileo \\ \tilev \\ \rlap{\tilei[\graycolor]}\tileo \\ \end{vmosaic}
\meq[.5em]{\xleftrightarrow[\rule{9mm}{0pt}]{\mathit{SI}_{\!2}}}
\raisebox{1.4cm}{\begin{vmosaic}[.8]{1}{3}{{x,y,z}}{{}}{{}}{{}} \rlap{\tilevi[densely dotted]}\tileiv\& \tilevi \& \rlap{\tilevi[dashed]}\tileiii \\ \end{vmosaic}}\hspace{2mm}
\begin{vmosaic}[.8]{3}{1}{{}}{{x,y,z}}{{}}{{}} \rlap{\tileiv[\graycolor]}\tileii \\ \tileix \\ \rlap{\tilei[\graycolor]}\tileiii \\ \end{vmosaic}
\\[3mm]
&\begin{vmosaic}[.8]{1}{3}{{a,x,x}}{{}}{{}}{{}} \rlap{\tileiv[densely dotted]}\tilevi \& \tileo \& \tileii[\graycolor] \\ \end{vmosaic}
\meq[.5em]{\xleftrightarrow[\rule{9mm}{0pt}]{\mathit{SI}_{\!3}}}
\begin{vmosaic}[.8]{1}{3}{{x,x,a}}{{}}{{}}{{}} \rlap{\tilev[densely dotted]}\tileii \& \tilev \& \rlap{\tileii[\graycolor]}\tileiv \\ \end{vmosaic}
\\[3mm]
&\begin{vmosaic}[.8]{1}{3}{{a,x,x}}{{}}{{}}{{}} \rlap{\tilei[\graycolor]}\rlap{\tileii[\graycolor]}\tilev[\graycolor] \& \rlap{\tilei[\graycolor]}\rlap{\tileii[\graycolor]}\tilev[\graycolor] \& \rlap{\tilei[\graycolor]}\rlap{\tileii[\graycolor]}\tilev[\graycolor] \\ \end{vmosaic}
\meq[.5em]{\xleftrightarrow[\rule{9mm}{0pt}]{\mathit{SI}_{\!4}}}
\begin{vmosaic}[.8]{1}{3}{{x,x,a}}{{}}{{}}{{}} \rlap{\tilei[\graycolor]}\rlap{\tileii[\graycolor]}\tilev[\graycolor] \& \rlap{\tilei[\graycolor]}\rlap{\tileii[\graycolor]}\tilev[\graycolor] \& \rlap{\tilei[\graycolor]}\rlap{\tileii[\graycolor]}\tilev[\graycolor] \\ \end{vmosaic}
\\[3mm]
&\begin{vmosaic}[.8]{1}{2}{{x,x}}{{a}}{{}}{{}} \tilei[\graycolor] \& \tileii \\ \end{vmosaic}
\meq[.5em]{\xleftrightarrow[\rule{9mm}{0pt}]{\mathit{SI}_{\!5}}}
\begin{vmosaic}[.8]{1}{2}{{a,x}}{{x}}{{}}{{}} \rlap{\tilei[\graycolor]}\tileiii \& \tilei \\ \end{vmosaic}
\\[3mm]
&\begin{vmosaic}[.8]{1}{2}{{x,x}}{{a}}{{}}{{}} \rlap{\tilei[\graycolor]}\rlap{\tileii[\graycolor]}\tilev[\graycolor] \& \tilei[\graycolor] \\ \end{vmosaic}
\meq[.5em]{\xleftrightarrow[\rule{9mm}{0pt}]{\mathit{SI}_{\!6}}}
\begin{vmosaic}[.8]{1}{2}{{a,x}}{{x}}{{}}{{}} \rlap{\tilei[\graycolor]}\rlap{\tileii[\graycolor]}\tilev[\graycolor] \& \tilei[\graycolor] \\ \end{vmosaic}
\\[3mm]
&\begin{vmosaic}[.8]{1}{2}{{x,x}}{{a}}{{}}{{}} \rlap{\tilev[densely dotted]}\tileii \& \tilev \\ \end{vmosaic}
\meq[.5em]{\xleftrightarrow[\rule{9mm}{0pt}]{\mathit{SI}_{\!7}}}
\begin{vmosaic}[.8]{1}{2}{{a,x}}{{x}}{{}}{{}} \rlap{\tileiv[densely dotted]}\tilevi \& \tileo \\ \end{vmosaic}
\\[3mm]
&\raisebox{1cm}{\begin{vmosaic}[.8]{1}{2}{{x,y}}{{}}{{}}{{}} \rlap{\tilei[\graycolor]}\tileiii \& \tilex \\ \end{vmosaic}}\hspace{2mm}
\begin{vmosaic}[.8]{2}{1}{{}}{{y,x}}{{}}{{}} \rlap{\tileiv[\graycolor]}\tileii \\ \tileviii \\ \end{vmosaic}
\meq[.5em]{\xleftrightarrow[\rule{9mm}{0pt}]{\mathit{SI}_{\!8}}}
\raisebox{1cm}{\begin{vmosaic}[.8]{1}{2}{{x,y}}{{}}{{}}{{}} \rlap{\tilei[\graycolor]}\tileiii \& \tilevii \\ \end{vmosaic}}\hspace{2mm}
\begin{vmosaic}[.8]{2}{1}{{}}{{y,x}}{{}}{{}} \rlap{\tileiv[\graycolor]}\tileii \\ \tileix \\ \end{vmosaic}
\\[3mm]
&\raisebox{1cm}{\begin{vmosaic}[.8]{1}{2}{{x,y}}{{}}{{}}{{}} \tilex \& \rlap{\tileii[\graycolor]}\tileiv \\ \end{vmosaic}}\hspace{2mm}
\begin{vmosaic}[.8]{2}{1}{{}}{{y,x}}{{}}{{}} \rlap{\tileiv[\graycolor]}\tileii \\ \tileviii \\ \end{vmosaic}
\meq[.5em]{\xleftrightarrow[\rule{9mm}{0pt}]{\mathit{SI}_{\!9}}}
\raisebox{1cm}{\begin{vmosaic}[.8]{1}{2}{{x,y}}{{}}{{}}{{}}  \tileviii \& \rlap{\tileii[\graycolor]}\tileiv  \\ \end{vmosaic}}\hspace{2mm}
\begin{vmosaic}[.8]{2}{1}{{}}{{y,x}}{{}}{{}} \rlap{\tileiv[\graycolor]}\tileii \\ \tilex \\ \end{vmosaic}
\end{align*}

\subsection{Stabilizations \& destabilizations}\label{stab-destab}

Just as with virtual knots and links viewed as knot diagrams on surfaces, we need to include certain stabilization and destabilization moves that allow us to represent virtual knots and links on surfaces of different genera.
\begin{align*}
&\begin{vmosaic}[.8]{1}{1}{{x}}{{y}}{{}}{{}} \tilevii[\graycolor] \\ \end{vmosaic}
\meq[.5em]{\xleftrightarrow[\rule{9mm}{0pt}]{\mathit{Stab}_1}}
\begin{vmosaic}[.8]{1}{1}{{y}}{{x}}{{}}{{}} \tilevii[\graycolor] \\ \end{vmosaic}
\\[3mm]
&\begin{vmosaic}[.8]{1}{2}{{x,y}}{{}}{{}}{{}} \tilevii[\graycolor] \& \tileviii[\graycolor] \\ \end{vmosaic}
\meq[.5em]{\xleftrightarrow[\rule{9mm}{0pt}]{\mathit{Stab}_2}}
\begin{vmosaic}[.8]{1}{2}{{y,x}}{{}}{{}}{{}} \tilevii[\graycolor] \& \tileviii[\graycolor] \\ \end{vmosaic}
\\[3mm]
&\begin{vmosaic}[.8]{1}{1}{{x}}{{y}}{{}}{{}} \rlap{\tileii[densely dotted]}\tilev \\ \end{vmosaic}
\meq[.5em]{\xleftrightarrow[\rule{9mm}{0pt}]{\mathit{Stab}_3}}
\begin{vmosaic}[.8]{1}{1}{{y}}{{x}}{{}}{{}} \rlap{\tilevi[densely dotted]}\tileiv \\ \end{vmosaic}
\\[3mm]
&\begin{vmosaic}[.8]{1}{2}{{x,y}}{{}}{{}}{{}} \rlap{\tileii[densely dotted]}\tilev \& \rlap{\tileii[\graycolor]}\tileiv \\ \end{vmosaic}
\meq[.5em]{\xleftrightarrow[\rule{9mm}{0pt}]{\mathit{Stab}_4}}
\begin{vmosaic}[.8]{1}{2}{{y,x}}{{}}{{}}{{}} \rlap{\tilevi[densely dotted]}\tileiv \& \tileii[\graycolor] \\ \end{vmosaic}
\end{align*}

\subsection{Mosaic injection \& ejection}\label{sec-inject}
In this section, we describe the process of enlarging or shrinking a mosaic without changing its genus or the link it represents. Let $\mathbb{V}^{(n)}$ denote the set of virtual $n$-mosaics. If $V^{(n)}\in\mathbb{V}^{(n)}$, we denote the $ij$-entry of $V^{(n)}$ by $V^{(n)}_{ij}$.

\begin{defn}\label{def:injection}
The \emph{standard virtual mosaic injection}
\begin{align*}
    \iota\colon\mathbb{V}^{(n)}&\to\mathbb{V}^{(n+2)}\\
    V^{(n)}&\mapsto V^{(n+2)}
\end{align*}
will be defined as
\[
V_{ij}^{(n+2)}=\begin{cases*}
V_{ij}^{(n)} & if $0\le i,j<n$\\
T_5 & if $i<n$, $j\ge n$, and $V_{i,n-1}^{(n)}\in\{T_2,T_3,T_5,T_7,T_8,T_9,T_{10}\}$ \\
T_6 & if $i\ge n$, $j<n$, and $V_{n-1,j}^{(n)}\in\{T_1,T_2,T_6,T_7,T_8,T_9,T_{10}\}$ \\
T_0 & $\vphantom{V_i^{(n)}}$otherwise,
\end{cases*}
\]
where the new boundary edges are labeled in adjacent pairs. The reverse process is called an \emph{ejection}.
\end{defn}

\begin{figure}[h]
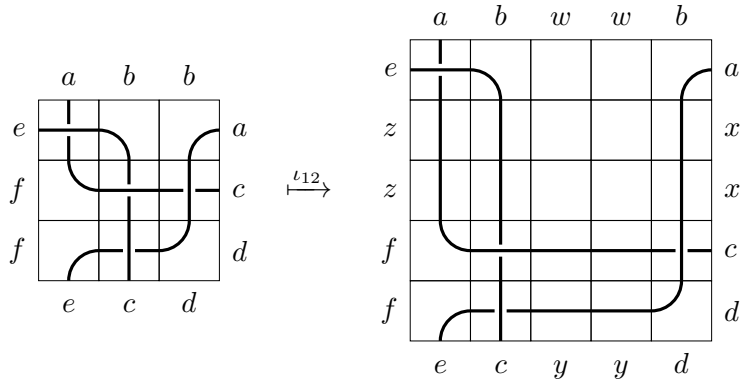

    \centering
\begin{vmosaic}[.8]{3}{3}{{a,b,b}}{{a,c,d}}{{d,c,e}}{{f,f,e}} 
\tileix \& \tilei \& \tileii \\
\tileiii \& \tileix \& \tilex \\
\tileii \& \tilex \& \tileiv \\
\end{vmosaic}
\meq[1em]{\xmapsto{\iota_{12}}}
\begin{vmosaic}[.8]{5}{5}{{a,b,w,w,b}}{{a,x,x,c,d}}{{d,y,y,c,e}}{{f,f,z,z,e}} 
\tileix \& \tilei \& \tileo \& \tileo \& \tileii \\
\tilevi \& \tilevi \& \tileo \& \tileo \& \tilevi \\
\tilevi \& \tilevi \& \tileo \& \tileo \& \tilevi \\
\tileiii \& \tileix \& \tilev \& \tilev \& \tilex \\
\tileii \& \tilex \& \tilev \& \tilev \& \tileiv \\
\end{vmosaic}
    \caption{A $(1,2)$-injection}
    \label{fig:inject}
\end{figure}

More generally, we can define the \emph{virtual mosaic $(i,j)$-injection} $\iota_{ij}$, in which the two additional rows are inserted below row $i$ (or at the top of the mosaic if $i=0$) and the two additional columns are inserted to the right of column $j$ (or at the far left of the mosaic if $j=0$). The standard mosaic injection is the same as $\iota_{nn}$. Topologically, the new adjacent pairs of edges are simply $S^2$ summands in the underlying surface. Thus, the new mosaic has the same genus as the original. Since the Gauss code is unaffected by the injection, the link represented by the mosaic is unchanged. Figure \ref{fig:inject} illustrates a $(1,2)$-injection on a mosaic of the classical figure-eight knot.

Here, we provide a pair of examples to illustrate how sequences of KL and virtual mosaic moves may be combined to achieve certain desirable results.

\begin{example}
In this example, we look at how to perform the following R1-like move that passes through the mosaic boundary. 

\[
\begin{vmosaic}[1]{1}{2}{{x,y}}{{x}}{{}}{{{}}} \tilex \& \tilev \\ \end{vmosaic}\meq[1em]{\longrightarrow}
\begin{vmosaic}[1]{1}{2}{{x,x}}{{y}}{{}}{{{}}} \tilei \& \tileo \\ \end{vmosaic}
\]

\bigskip

\noindent We begin with move (1), which is an application of $\mathit{Stab}_3$.
\[
\begin{vmosaic}[1]{1}{2}{{x,y}}{{x}}{{}}{{{}}} \tilex \& \tilev \\ \end{vmosaic}\meq[1em]{\xrightarrow{(1)}}
\begin{vmosaic}[1]{1}{2}{{x,x}}{{y}}{{}}{{{}}} \tilex \& \tileiv \\ \end{vmosaic}\]
Move (2) is an injection introducing new labels $z$ and $w$, and move (3) is surface isotopy $\mathit{SI}_{\!1}$.
\[
\meq[1em]{\xrightarrow{(2)}}
\begin{vmosaic}[1]{3}{4}{{x,x, z,z}}{{w,w,y}}{{}}{{{}}} \tilevi \& \tilevi \& \tileo \& \tileo\\
\tilevi \& \tilevi \& \tileo \& \tileo\\
\tilex \& \tileiv \& \tileo \& \tileo\\ \end{vmosaic}
\meq[1em]{\xrightarrow{(3)}}
\begin{vmosaic}[1]{3}{4}{{x,x, z,z}}{{w,w,y}}{{}}{{{}}} \tileii \& \tilei \& \tileo \& \tileo\\
\tilevi \& \tilevi \& \tileo \& \tileo\\
\tilex \& \tileiv \& \tileo \& \tileo\\ \end{vmosaic}
\]
Move (4) is KL isotopy $\mathit{P}_6$, while move (5) is the KL Reidemeister 1 move, $\mathit{R}_1/\mathit{R}_{1'}$.
\[
\meq[1em]{\xrightarrow{(4)}}
\begin{vmosaic}[1]{3}{4}{{x,x, z,z}}{{w,w,y}}{{}}{{{}}} \tileo \& \tileo \& \tileo \& \tileo\\
\tileii \& \tilei \& \tileo \& \tileo\\
\tilex \& \tileiv \& \tileo \& \tileo\\ \end{vmosaic}\meq[1em]{\xrightarrow{(5)}}
\begin{vmosaic}[1]{3}{4}{{x,x, z,z}}{{w,w,y}}{{}}{{{}}} \tileo \& \tileo \& \tileo \& \tileo\\
\tileo \& \tileo \& \tileo \& \tileo\\
\tilei \& \tileo \& \tileo \& \tileo\\ \end{vmosaic}\]
Finally, move (6) is an ejection, and we are done.
\[\meq[1em]{\xrightarrow{(6)}}
\begin{vmosaic}[1]{1}{2}{{x,x}}{{y}}{{}}{{{}}} \tilei \& \tileo \\ \end{vmosaic}
\]
\end{example}\bigskip

\begin{example}
Here, we demonstrate that the two virtual mosaics of the virtual trefoil pictured in Figure \ref{diff_genus_vtref} are connected by a sequence of virtual mosaic moves. Begin with the genus-2 mosaic in Figure \ref{diff_genus_vtref}(R). Move (1) represents $\mathit{Stab}_1$ applied to the virtual trefoil at the northeast corner. The new mosaic has genus 1.
\[
\begin{vmosaic}[1]{2}{2}{{a,b}}{{a,b}}{{c,d}}{{c,d}}
\tilex \& \tilevii \\
\tileix \& \tileviii \\
\end{vmosaic}\meq[1em]{\xrightarrow{(1)}}
\begin{vmosaic}[1]{2}{2}{{a,a}}{{b,b}}{{c,d}}{{c,d}}
\tilex \& \tilevii \\
\tileix \& \tileviii \\
\end{vmosaic}
\]
Move (2) is a standard injection, introducing new labels $e,f,g,h$, and (3) is an instance of surface isotopy $\mathit{SI}_{\!1}$.
\[
\meq[1em]{\xrightarrow{(2)}}
\begin{vmosaic}[.8]{4}{4}{{a,a,e,e}}{{b,b,f,f}}{{g,g,c,d}}{{h,h,c,d}} 
\tilex \& \tilevii \& \tilev \& \tilev \\
\tileix \& \tileviii \&  \tilev \& \tilev \\
\tilevi \& \tilevi \& \tileo \& \tileo\\
\tilevi \& \tilevi \&  \tileo \& \tileo\\
\end{vmosaic}
\meq[1em]{\xrightarrow{(3)}}
\begin{vmosaic}[.8]{4}{4}{{a,a,e,e}}{{b,b,f,f}}{{g,g,c,d}}{{h,h,c,d}} 
\tilex \& \tilevii \& \tilev \& \tilei \\
\tileix \& \tileviii \&  \tilev \& \tileiv \\
\tilevi \& \tilevi \& \tileo \& \tileo\\
\tilevi \& \tilevi \&  \tileo \& \tileo\\
\end{vmosaic}
\]
Move (4) is surface isotopy $\mathit{SI}_{\!3}$, while (5) is KL planar isotopy $\mathit{P}_1$. 
\[
\meq[1em]{\xrightarrow{(4)}}
\begin{vmosaic}[.8]{4}{4}{{a,e,e,a}}{{b,b,f,f}}{{g,g,c,d}}{{h,h,c,d}} 
\tilex \& \tilei \& \tileo \& \tilevi \\
\tileix \& \tileviii \&  \tilev \& \tileiv \\
\tilevi \& \tilevi \& \tileo \& \tileo\\
\tilevi \& \tilevi \&  \tileo \& \tileo\\
\end{vmosaic}
\meq[1em]{\xrightarrow{(5)}}
\begin{vmosaic}[.8]{4}{4}{{a,e,e,a}}{{b,b,f,f}}{{g,g,c,d}}{{h,h,c,d}} 
\tilex \& \tilei \& \tileii \& \tileiv \\
\tileix \& \tileviii \&  \tileiv \& \tileo \\
\tilevi \& \tilevi \& \tileo \& \tileo\\
\tilevi \& \tilevi \&  \tileo \& \tileo\\
\end{vmosaic}
\]
Move (6) is surface isotopy $\mathit{SI}_{\!5}$, while (7) is KL planar isotopy $\mathit{P}_5$.
\[
\meq[1em]{\xrightarrow{(6)}}
\begin{vmosaic}[.8]{4}{4}{{a,e,e,b}}{{b,a,f,f}}{{g,g,c,d}}{{h,h,c,d}} 
\tilex \& \tilei \& \tileii \& \tilei \\
\tileix \& \tileviii \&  \tileiv \& \tileiii \\
\tilevi \& \tilevi \& \tileo \& \tileo\\
\tilevi \& \tilevi \&  \tileo \& \tileo\\
\end{vmosaic}
\meq[1em]{\xrightarrow{(7)}}
\begin{vmosaic}[.8]{4}{4}{{a,e,e,b}}{{b,a,f,f}}{{g,g,c,d}}{{h,h,c,d}} 
\tilex \& \tilei \& \tileo \& \tileo \\
\tileix \& \tileviii \&  \tilev \& \tilev \\
\tilevi \& \tilevi \& \tileo \& \tileo\\
\tilevi \& \tilevi \&  \tileo \& \tileo\\
\end{vmosaic}
\]
Finally, move (8) is surface isotopy $\mathit{SI}_{\!4}$ applied along the top right edge of the virtual mosaic. Note that this surface isotopy can be applied to blank tiles as well as tiles containing portions of the knot that don't pass through the boundary, with the effect in either case that pairs of identical edge labels can be moved past other edge labels. Move (9) is the ejection that returns us to the 2-mosaic in Figure \ref{diff_genus_vtref}(L).
\[
\meq[1em]{\xrightarrow{(8)}}
\begin{vmosaic}[.8]{4}{4}{{a,b,e,e}}{{b,a,f,f}}{{g,g,c,d}}{{h,h,c,d}} 
\tilex \& \tilei \& \tileo \& \tileo \\
\tileix \& \tileviii \&  \tilev \& \tilev \\
\tilevi \& \tilevi \& \tileo \& \tileo\\
\tilevi \& \tilevi \&  \tileo \& \tileo\\
\end{vmosaic}
\meq[1em]{\xrightarrow{(9)}}
\begin{vmosaic}[1]{2}{2}{{a,b}}{{b,a}}{{c,d}}{{c,d}}
\tilex \& \tilei \\
\tileix \& \tileviii \\
\end{vmosaic}
\]
\end{example}

\if false 
\begin{example}
In this example, we demonstrate that the virtual trefoil, pictured on the top left on a genus 2 surface, can also be realized on a genus 1 surface. This is the result of a sequence of planar and surface isotopies, a stabilization, an injection, and an ejection. 

Specifically, (1) represents a stabilization (\ref{stab-destab}(\ref{stab-cornerswitch})) of the virtual trefoil that occurs at the northeast corner.
\[
\begin{vmosaic}[1]{2}{2}{{a,b}}{{a,b}}{{c,d}}{{c,d}}
\tilex \& \tilevii \\
\tileix \& \tileviii \\
\end{vmosaic}\meq[1em]{\xrightarrow{(1)}}
\begin{vmosaic}[1]{2}{2}{{a,a}}{{b,b}}{{c,d}}{{c,d}}
\tilex \& \tilevii \\
\tileix \& \tileviii \\
\end{vmosaic}
\]
Move (2) is an injection, and (3) is an instance of surface isotopy \ref{surf-iso}(\ref{iso-edgepinch}).
\[
\meq[1em]{\xrightarrow{(2)}}
\begin{vmosaic}[.8]{4}{4}{{a,a,e,e}}{{f,f,b,b}}{{h,h,c,d}}{{c,d,g,g}} 
\tilevi \& \tilevi \& \tileo \& \tileo \\
\tilevi \& \tilevi \&  \tileo \& \tileo \\
\tilex \& \tilevii \& \tilev \& \tilev\\
\tileix \& \tileviii \&  \tilev \& \tilev\\
\end{vmosaic}
\meq[1em]{\xrightarrow{(3)}}
\begin{vmosaic}[.8]{4}{4}{{a,a,e,e}}{{f,f,b,b}}{{h,h,c,d}}{{c,d,g,g}} 
\tilevi \& \tilevi \& \tileo \& \tileo \\
\tilevi \& \tilevi \&  \tileo \& \tileo \\
\tilex \& \tilevii \& \tilev \& \tilei\\
\tileix \& \tileviii \&  \tilev \& \tileiv\\
\end{vmosaic}
\]
Move (4) is an instance of surface isotopy \ref{surf-iso}(\ref{iso-bumpslide}), while Move (5) is given by surface isotopy \ref{surf-iso}(\ref{iso-cornerbumpslide}).
\[
\meq[1em]{\xrightarrow{(4)}}
\begin{vmosaic}[.8]{4}{4}{{a,e,e,a}}{{f,f,b,b}}{{h,h,c,d}}{{c,d,g,g}} 
\tilevi \& \tileii \& \tilev \& \tileiv \\
\tilevi \& \tilevi \&  \tileo \& \tileo \\
\tilex \& \tilevii \& \tilev \& \tilei\\
\tileix \& \tileviii \&  \tilev \& \tileiv\\
\end{vmosaic}
\meq[1em]{\xrightarrow{(5)}}
\begin{vmosaic}[.8]{4}{4}{{a,a,e,f}}{{f,a,b,b}}{{h,h,c,d}}{{c,d,g,g}} 
\tilevi \& \tileii \& \tilev \& \tilei \\
\tilevi \& \tilevi \&  \tileo \& \tileiii \\
\tilex \& \tilevii \& \tilev \& \tilei\\
\tileix \& \tileviii \&  \tilev \& \tileiv\\
\end{vmosaic}
\]
The next move, Move (6), is actually a sequence of planar isotopies, $P_1$, $P_2$, and $P_3$. Move (7) is an instance of surface isotopy \ref{surf-iso}(\ref{iso-bumpslide}).
\[
\meq[1em]{\xrightarrow{(6)}}
\begin{vmosaic}[.8]{4}{4}{{a,a,e,f}}{{f,a,b,b}}{{h,h,c,d}}{{c,d,g,g}} 
\tilevi \& \tileo \& \tileo \& \tileo \\
\tilevi \& \tileo \&  \tileii \& \tilev \\
\tilex \& \tilei \& \tilevi \& \tileo\\
\tileix \& \tileviii \&  \tileiv \& \tileo\\
\end{vmosaic}
\meq[1em]{\xrightarrow{(7)}}
\begin{vmosaic}[.8]{4}{4}{{a,a,e,f}}{{f,b,b,a}}{{h,h,c,d}}{{c,d,g,g}} 
\tilevi \& \tileo \& \tileo \& \tileo \\
\tilevi \& \tileo \&  \tileii \& \tilei \\
\tilex \& \tilei \& \tilevi \& \tilevi\\
\tileix \& \tileviii \&  \tileiv \& \tileiii\\
\end{vmosaic}
\]
Finally, move (8) is a sequence of planar isotopies $P_6$ and $P_5$ followed by two instances of surface isotopy \ref{surf-iso}(\ref{iso-bumpslide}). Note that this surface isotopy can be applied to blank tiles as well as tiles containing portions of the knot, with the effect in either case that pairs of identical edge labels can be moved past other edge labels. Move (9) is the ejection that returns us to a 2 by 2 mosaic, but this time, our mosaic has genus 1.
\[
\meq[1em]{\xrightarrow{(8)}}
\begin{vmosaic}[1]{4}{4}{{a,b,e,e}}{{f,f,b,a}}{{h,h,c,d}}{{c,d,g,g}} 
\tilevi \& \tileo \& \tileo \& \tileo \\
\tilevi \& \tileo \&  \tileo \& \tileo \\
\tilex \& \tilei \& \tileo \& \tileo\\
\tileix \& \tileviii \&  \tilev \& \tilev\\
\end{vmosaic}
\meq[1em]{\xrightarrow{(9)}}
\begin{vmosaic}[1]{2}{2}{{a,b}}{{b,a}}{{c,d}}{{c,d}}
\tilex \& \tilei \\
\tileix \& \tileviii \\
\end{vmosaic}
\]
\end{example}
\fi 

\section{Relationship with Virtual Knot Theory}

Since virtual knots can be viewed as equivalence classes of knot diagrams on orientable surfaces, it seems natural to ask if virtual mosaic theory is equivalent to virtual knot theory. We have a partial answer to this question.

\begin{thm} If $L$ is a virtual knot or link, then there is a virtual mosaic that represents $L$.
\end{thm}

\begin{proof}
Let $L$ be a virtual knot or link and $B(L)$ be a virtual braid diagram whose closure has virtual knot or link type $L$, as in \cite{braids}.  We rotate $B(L)$ by 45 degrees and place the corresponding braid generators into a grid as in Figure \ref{fig:b-gen}. 
\begin{figure}[htb]
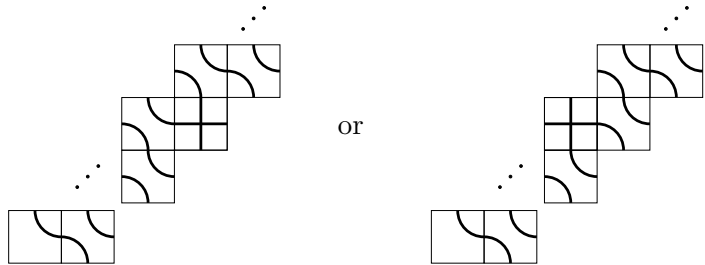

    \centering
 \begin{mosaic}[.7]
  \& \& \& \&  \& \tileuddots \\
  \& \& \& \& \tilevii \& \tilevii \\
  \& \& \& \tilevii \& \rlap{\tilev}\tilevi \\
  \& \tileuddots \& \& \tilevii \\[1mm]
  \tileiii \& \tilevii \& \hspace{1mm} \& \\ 
 \end{mosaic} \meq[2em]{\text{or}}
 \begin{mosaic}[.7]
  \& \& \& \& \& \tileuddots \\
  \& \& \& \& \tilevii \& \tilevii \\
  \& \& \& \rlap{\tilev}\tilevi \& \tilevii \\
  \& \tileuddots \& \& \tilevii \\[1mm]
  \tileiii \& \tilevii \& \hspace{1mm} \\ 
 \end{mosaic}
    \caption{A braid generator inside a mosaic.}
    \label{fig:b-gen}
\end{figure}
Crossings may be classical or virtual. We extend the braid strands to the boundary and identify boundary edges according to standard closure rules. The process is illustrated in Figure \ref{fig:b-virt} with the closed virtual braid $\sigma_2^{-1}\sigma_3v_2\sigma_1^{-1}$. Remaining boundary edges may be identified in pairs.

\begin{figure}[htb]
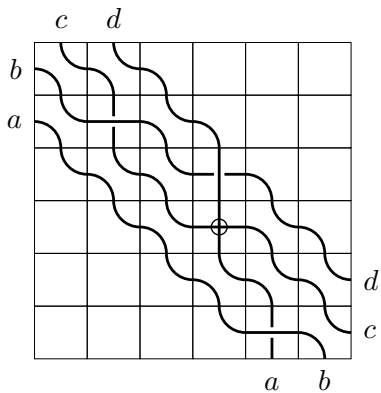

    \centering
 \begin{vmosaic}[.7]{6}{6}{{c,d}}{{,,,,d,c}}{{b,a}}{{,,,,a,b}}
  \tilevii \& \tilevii \& \tilei \& \tileo \& \tileo \& \tileo  \\
  \tilevii \& \tileix \& \tilevii \& \tilei \& \tileo \& \tileo \\
  \tileiii \& \tilevii \& \tilevii \& \tilex \& \tilei \& \tileo  \\
  \tileo \& \tileiii \& \tilevii \& \tilec \& \tilevii \& \tilei  \\
  \tileo \& \tileo \& \tileiii \& \tilevii \& \tilevii \& \tilevii  \\
  \tileo \& \tileo \& \tileo \& \tileiii \& \tileix \& \tilevii  \\
 \end{vmosaic} 
    \caption{A closed braid with virtual crossings.}
    \label{fig:b-virt}
\end{figure}

The obstacle we now face is that some of the tiles in our mosaic may be virtual crossing tiles, which are not permitted on a virtual mosaic. Our goal, then, is to slide these virtual crossings off of the mosaic board so that they are represented only implicitly in the surface.

We begin with the top-most crossing in the braid. If it is a virtual crossing, we replace the crossing tile with tile $T_7$ 
and swap the labels associated to the crossing strands at the top-left of the grid, as in Figure \ref{fig:b-swap}. 
\begin{figure}[htb]
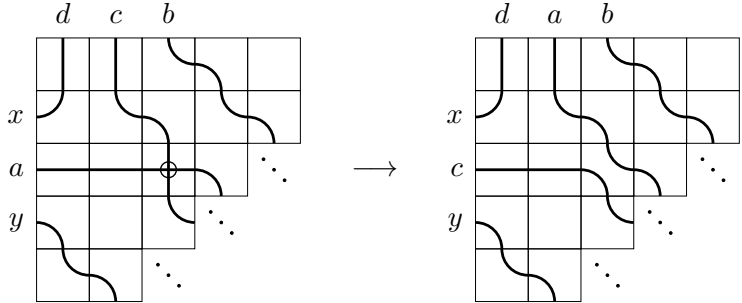

    \centering
 \begin{vmosaic}[.7]{5}{5}{{d,c,b}}{{{}}}{{}}{{,y,a,x}}
  \tilevi \& \tilevi \& \tileiii \& \tilei \& \tileo  \\
  \tileiv \& \tileiii \& \tilei \& \tileiii \& \tilei  \\
  \tilev \& \tilev \& \tilec \& \tilei \& \tiledddots \\
  \tilei \& \tileo \& \tileiii \& \tiledddots \\
  \tileiii \& \tilei \& \tiledddots \\
 \end{vmosaic}\meq[1em]{\longrightarrow}
  \begin{vmosaic}[.7]{5}{5}{{d,a,b}}{{{}}}{{}}{{,y,c,x}}
  \tilevi \& \tilevi \& \tileiii \& \tilei \& \tileo  \\
  \tileiv \& \tileiii \& \tilei \& \tileiii \& \tilei  \\
  \tilev \& \tilev \& \tilevii \& \tilei \& \tiledddots \\
  \tilei \& \tileo \& \tileiii \& \tiledddots \\
  \tileiii \& \tilei \& \tiledddots \\
 \end{vmosaic}
    \caption{Moving a virtual crossing off the board.}
    \label{fig:b-swap}
\end{figure}
Such an operation does not affect the Gauss code of our knot or link, and therefore preserves the virtual knot/link type. 

If the topmost crossing is a classical crossing, we perform KL isotopies, possibly along with some number of surface isotopy moves and injections, to move the crossing tile so that it is along the boundary. Let's say the edge of the crossing tile is labeled $b$. We perform the following sequence of moves:
\begin{itemize}
    \item injections on both columns (or rows) adjacent to $b$
    \item injections on both columns (or rows) adjacent to the other $b$ label
    \item injections on both boundary edges containing $b$ labels.
    \item KL isotopies to move the crossing back to the edge.
\end{itemize}
The result appears on the left in Figure \ref{fig:crossingpush}. 
\begin{figure}[htb]
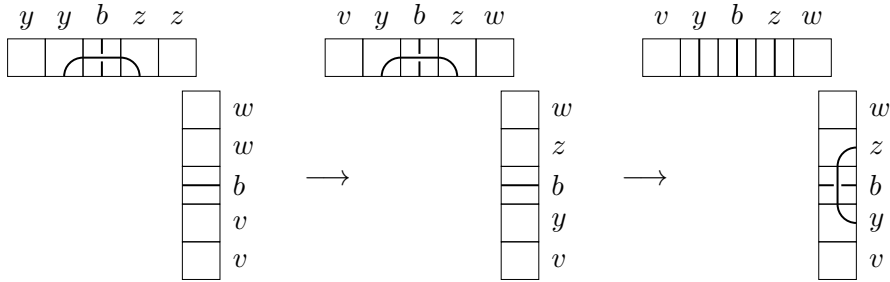

    \centering
 \begin{tabular}{p{3.8cm}p{3.8cm}p{3.8cm}}
 \begin{tabular}{r} 
 \begin{vmosaic}[.5]{1}{5}{{y,y,b,z,z}}{{}}{{}}{{}}
  \tileo \& \tileii \& \tileix \& \tilei \& \tileo \\
 \end{vmosaic}\rule{9mm}{0pt}\\[3mm]
 \begin{vmosaic}[.5]{5}{1}{{}}{{w,w,b,v,v}}{{}}{{}}
   \tileo \\ \tileo \\ \tilev \\ \tileo \\ \tileo \\
 \end{vmosaic}
 \end{tabular}
 &
 \begin{tabular}{r} 
 \begin{vmosaic}[.5]{1}{5}{{v,y,b,z,w}}{{}}{{}}{{}}
  \tileo \& \tileii \& \tileix \& \tilei \& \tileo \\
 \end{vmosaic}\rule{9mm}{0pt}\\[3mm]
 \llap{$\longrightarrow$}\hspace{2cm}\begin{vmosaic}[.5]{5}{1}{{}}{{w,z,b,y,v}}{{}}{{}}
   \tileo \\ \tileo \\ \tilev \\ \tileo \\ \tileo \\
 \end{vmosaic}
 \end{tabular}
 &
  \begin{tabular}{r} 
 \begin{vmosaic}[.5]{1}{5}{{v,y,b,z,w}}{{}}{{}}{{}}
  \tileo \& \tilevi \& \tilevi \& \tilevi \& \tileo \\
 \end{vmosaic}\rule{9mm}{0pt}\\[3mm]
 \llap{$\longrightarrow$}\hspace{2cm}\begin{vmosaic}[.5]{5}{1}{{}}{{w,z,b,y,v}}{{}}{{}}
   \tileo \\ \tileii \\ \tilex \\ \tileiii \\ \tileo \\
 \end{vmosaic}
 \end{tabular}
 \end{tabular}
    \caption{Moving a classical crossing off the board.}
    \label{fig:crossingpush}
\end{figure}
Now, since the tiles two away from the crossing must be empty, we may swap those labels with the labels adjacent to the other $b$ without changing the Gauss code, producing the second image in Figure \ref{fig:crossingpush}. (Note that this operation is not explicitly in our list of virtual mosaic moves, but all that matters here is that the resulting knot has the same virtual knot type.) We may now perform a surface isotopy of type $\mathit{SI}_{\!2}$ to move the crossing off the top of the braid, as shown in the third image in Figure \ref{fig:crossingpush}.

From here, we repeat this process of removing virtual crossing tiles and taking classical crossings from the top-left and moving them to the bottom-right until no more virtual crossing tiles appear on the mosaic. We have thus produced a virtual mosaic. Furthermore, the resulting knot or link has the same knot/link type as $L$.
\end{proof}

We end our discussion of the relationship between virtual knots and virtual mosaics with a conjecture.

\begin{conj}
Virtual knot theory is equivalent to virtual mosaic theory. That is, two virtual knots or links, $L_1$ and $L_2$, are equivalent if and only if any two virtual mosaics that represent $L_1$ and $L_2$ are equivalent.
\end{conj}

\section{Conclusion}

While we have made progress in the study of virtual mosaic knots, there are many more interesting open questions that can be studied. Here, in conclusion, we list a number of our favorites.

\begin{quest} If $\vmn{L}=n$, and the minimum genus among $n$-mosaics representing $L$ is $g$, is $g$ the genus of the knot? In particular, if $K$ is a classical knot, is $\vmn{K}$ always realized on a genus-0 mosaic? The answer is yes for classical knots with crossing number $\leq 8$, for 9-crosssing classical knots with $n=3$, and for virtual knots with $n\leq 2$. (The latter result was verified using \cite{adams}.) 
\end{quest}


\begin{quest} What is the relationship between virtual mosaic number and mosaic number for classical knots? By Proposition \ref{class-virt-rel}, $\vmn{K}\le m(K)-2$ for nontrivial $K$, but sometimes this inequality is strict. See Example \ref{ex:7_2}, for instance. If equality does not always hold, is there a fixed integer $c$ for which $\vmn{K}+c\ge \mn{K}-2$? 
\end{quest}

\begin{quest}Knot $7_1$ is an example of a knot where the virtual mosaic number is realized on a non-reduced projection  (i.e., a projection with more crossings than the crossing number). Knots $7_3$, $8_7$, $8_{10}$ and $8_{19}$ are also examples illustrating this phenomenon. Is there an infinite family of such examples?
\end{quest}

\begin{quest} Which links with crossing number $\leq 9$ have $\vmn{L} = 3$?
\end{quest}

\begin{quest} For a given $n$, how many distinct virtual knots (or links) can be represented on a virtual $n$ mosaic?
\end{quest}

\begin{quest} How can we detect virtual mosaic unknots? Are there virtual mosaic ``culprits" where mosaics need to be made more complex (for instance, via injections) before these unknot representations can be simplified?
\end{quest}

\begin{quest} Which nontrivial tile patterns have edge identifications that yield the unknot?
\end{quest}

\begin{quest}
Given a tile pattern, what is the probability of getting an unknot with a random choice of genus 0, 1-component edge labelings? 
For a given tile pattern, what is the distribution of knot types that come from genus 0, 1-component edge labelings.
\end{quest}

\begin{quest} Can we generalize tile patterns for any infinite knot families to determine nontrivial bounds on $\vmn{K}$? (Note: $n$ is an upper bound for $\vmn{T_{2,p}}$ for $p=\lceil\frac{n^2}{2}\rceil$.)
\end{quest}

\begin{quest} If each mosaic tile is given a weight and the \emph{weight} of a virtual $n$-mosaic is the sum of the weights of its $n^2$ tiles, what is the relationship between a given knot's minimum weight and its virtual mosaic number?
\end{quest}

\appendix
\section{Classical Knots with Eight or Fewer Crossings and Minimal Virtual Mosaics}\label{minimal-mosaics}
\noindent
\resizebox{\textwidth}{!}{%
\begin{tabular}{cccc}
\begin{vmosaic}[2.4]{1}{1}{{a}}{{a}}{{b}}{{b}} 
\tilei \\
\end{vmosaic} &  
\begin{vmosaic}[1.2]{2}{2}{{a,b}}{{b,a}}{{c,c}}{{d,d}} 
\tilex \& \tilei \\
\tileix \& \tilex \\
\end{vmosaic} &
\begin{vmosaic}[.8]{3}{3}{{a,b,b}}{{a,c,d}}{{d,c,e}}{{f,f,e}} 
\tileix \& \tilei \& \tileii \\
\tileiii \& \tileix \& \tilex \\
\tileii \& \tilex \& \tileiv \\
\end{vmosaic} &
\begin{vmosaic}[.8]{3}{3}{{a,b,c}}{{d,e,e}}{{d,c,b}}{{a,f,f}} 
\tilex \& \tilevii \& \tilex \\
\tilevii \& \tilex \& \tilevii \\
\tilex \& \tilevii \& \tilex \\
\end{vmosaic} \\
$0_1$ & $3_1$ & $4_1$ & $5_1$ 
\end{tabular}
}

\medskip\noindent
\resizebox{\textwidth}{!}{%
\begin{tabular}{cccc}
\begin{vmosaic}[.8]{3}{3}{{a,a,b}}{{c,d,d}}{{e,e,c}}{{b,f,f}} 
\tilex \& \tilevii \& \tilex \\
\tilevii \& \tilex \& \tilevii \\
\tilex \& \tilevii \& \tilex \\
\end{vmosaic} &
\begin{vmosaic}[.8]{3}{3}{{a,b,b}}{{a,c,c}}{{d,e,e}}{{d,f,f}} 
\tilex \& \tileix \& \tilex \\
\tileix \& \tilevii \& \tilevii \\
\tilex \& \tileviii \& \tilex \\
\end{vmosaic} &
\begin{vmosaic}[.8]{3}{3}{{a,b,c}}{{c,d,e}}{{e,d,b}}{{a,f,f}} 
\tilex \& \tileix \& \tilei \\
\tileix \& \tilevii \& \tileix \\
\tilex \& \tileix \& \tileiv \\
\end{vmosaic} &
\begin{vmosaic}[.8]{3}{3}{{a,a,b}}{{c,c,b}}{{d,e,f}}{{f,e,d}} 
\tileix \& \tilex \& \tileix \\
\tileiii \& \tileix \& \tilex \\
\tileo \& \tileiii \& \tileix \\
\end{vmosaic} \\
$5_2$ & $6_1$ & $6_2$ & $6_3$
\end{tabular}
}

\medskip\noindent
\resizebox{\textwidth}{!}{%
\begin{tabular}{cccc}
\begin{vmosaic}[.8]{3}{3}{{a,b,c}}{{d,e,e}}{{d,c,b}}{{a,f,f}} 
\tileix \& \tileix \& \tileix \\
\tilex \& \tilex \& \tilex \\
\tileix \& \tileix \& \tileix \\
\end{vmosaic} &
\begin{vmosaic}[.8]{3}{3}{{a,b,b}}{{c,d,d}}{{c,e,e}}{{a,f,f}} 
\tilex \& \tileix \& \tilex \\
\tileix \& \tileviii \& \tilevii \\
\tilex \& \tileix \& \tilex \\
\end{vmosaic} &
\begin{vmosaic}[.8]{3}{3}{{a,a,b}}{{c,d,d}}{{e,e,f}}{{f,c,b}} 
\tilex \& \tileix \& \tilex \\
\tilex \& \tileix \& \tilex \\
\tileiii \& \tilex \& \tileix \\
\end{vmosaic} &
\begin{vmosaic}[.8]{3}{3}{{a,a,b}}{{b,c,c}}{{d,d,e}}{{e,f,f}} 
\tilex \& \tileix \& \tilei \\
\tileix \& \tilex \& \tileix \\
\tileiii \& \tileix \& \tilex \\
\end{vmosaic} \\
$7_1$ & $7_2$ & $7_3$ & $7_4$
\end{tabular}
}

\medskip\noindent
\resizebox{\textwidth}{!}{%
\begin{tabular}{cccc}
\begin{vmosaic}[.8]{3}{3}{{a,a,b}}{{c,c,b}}{{d,d,e}}{{f,f,e}} 
\tileviii \& \tilevii \& \tileix \\
\tilex \& \tileix \& \tilex \\
\tileix \& \tilex \& \tileix \\
\end{vmosaic} &
\begin{vmosaic}[.8]{3}{3}{{a,a,b}}{{b,c,d}}{{e,f,f}}{{e,d,c}} 
\tileix \& \tilex \& \tilei \\
\tilex \& \tileviii \& \tilex \\
\tileix \& \tilex \& \tileix \\
\end{vmosaic} &
\begin{vmosaic}[.8]{3}{3}{{a,a,b}}{{b,c,d}}{{e,e,d}}{{c,f,f}} 
\tileix \& \tilex \& \tilei \\
\tilevii \& \tileix \& \tilex \\
\tileix \& \tilex \& \tileix \\
\end{vmosaic} &
\begin{vmosaic}[.6]{4}{4}{{a,a,b,b}}{{c,c,d,e}}{{e,d,f,f}}{{g,g,h,h}} 
\tileix \& \tilevii \& \tileix \& \tilex \\
\tilevii \& \tileix \& \tileviii \& \tileix \\
\tileix \& \tileviii \& \tileiv \& \tileiii \\
\tilex \& \tileix \& \tilei \& \tileo \\
\end{vmosaic} \\
$7_5$ & $7_6$ & $7_7$ & $8_1$
\end{tabular}
}

\medskip\noindent
\resizebox{\textwidth}{!}{%
\begin{tabular}{cccc}
\begin{vmosaic}[.6]{4}{4}{{a,a,b,c}}{{d,e,f,f}}{{e,d,c,b}}{{g,h,h,g}} 
\tilex \& \tilevii \& \tilex \& \tileix \\
\tileiii \& \tilex \& \tilevii \& \tilex \\
\tileii \& \tilevii \& \tilex \& \tilevii \\
\tileix \& \tilevii \& \tilevii \& \tilex \\
\end{vmosaic} &
\begin{vmosaic}[.6]{4}{4}{{a,a,b,c}}{{d,d,e,e}}{{f,g,g,f}}{{c,h,h,b}} 
\tileix \& \tilevii \& \tilevii \& \tilex \\
\tilex \& \tileviii \& \tileviii \& \tileix \\
\tileix \& \tilevii \& \tilevii \& \tilex \\
\tilex \& \tileiv \& \tileiii \& \tileix \\
\end{vmosaic} &
\begin{vmosaic}[.6]{4}{4}{{a,b,b,a}}{{c,d,d,c}}{{e,e,f,f}}{{g,g,h,h}} 
\tileii \& \tilex \& \tileix \& \tilei \\
\tilevi \& \tilevi \& \tileiii \& \tileix \\
\tilex \& \tileix \& \tilev \& \tilex \\
\tileix \& \tilex \& \tilev \& \tileiv \\
\end{vmosaic} &
\begin{vmosaic}[.8]{3}{3}{{a,b,b}}{{c,c,a}}{{d,e,e}}{{f,f,d}} 
\tileix \& \tilex \& \tileix \\
\tilex \& \tileviii \& \tilex \\
\tileix \& \tilex \& \tileix \\
\end{vmosaic} \\
$8_2$ & $8_3$ & $8_4$ & $8_5$
\end{tabular}
}

\medskip\noindent
\resizebox{\textwidth}{!}{%
\begin{tabular}{cccc}
\begin{vmosaic}[.6]{4}{4}{{a,a,b,b}}{{c,d,e,e}}{{f,f,g,g}}{{d,h,h,c}} 
\tilex \& \tilevii \& \tilei \& \tileii \\
\tileix \& \tileviii \& \tileix \& \tilex \\
\tilex \& \tileix \& \tilevii \& \tileiv \\
\tileix \& \tilex \& \tileiv \& \tileo \\
\end{vmosaic} &
\begin{vmosaic}[.8]{3}{3}{{a,a,b}}{{c,d,e}}{{f,f,e}}{{d,c,b}} 
\tilex \& \tileix \& \tilex \\
\tilex \& \tileix \& \tilex \\
\tileix \& \tilex \& \tileix \\
\end{vmosaic} &
\begin{vmosaic}[.8]{3}{3}{{a,b,b}}{{a,c,c}}{{d,e,e}}{{f,f,d}} 
\tileix \& \tilex \& \tileix \\
\tilex \& \tileviii \& \tilex \\
\tileix \& \tilex \& \tileix \\
\end{vmosaic} &
\begin{vmosaic}[.6]{4}{4}{{a,a,b,c}}{{c,d,e,e}}{{f,f,g,g}}{{d,h,h,b}} 
\tilex \& \tilevii \& \tilevii \& \tilei \\
\tileix \& \tileviii \& \tileviii \& \tileviii \\
\tilex \& \tileix \& \tilex \& \tileix \\
\tileiv \& \tileiii \& \tileix \& \tilex \\
\end{vmosaic} \\
$8_6$ & $8_7$ & $8_8$ & $8_9$
\end{tabular}
}

\medskip\noindent
\resizebox{\textwidth}{!}{%
\begin{tabular}{cccc}
\begin{vmosaic}[.8]{3}{3}{{a,a,b}}{{c,d,d}}{{e,e,c}}{{b,f,f}} 
\tilex \& \tileix \& \tileix \\
\tileix \& \tileix \& \tilex \\
\tileix \& \tilex \& \tileix \\
\end{vmosaic} &
\begin{vmosaic}[.6]{4}{4}{{a,a,b,b}}{{c,c,d,d}}{{e,e,f,f}}{{g,h,h,g}} 
\tilex \& \tilevii \& \tilei \& \tileo \\
\tileix \& \tileviii \& \tileix \& \tilei \\
\tilex \& \tileix \& \tilex \& \tileiv \\
\tileix \& \tilex \& \tileiv \& \tileo \\
\end{vmosaic} &
\begin{vmosaic}[.8]{3}{3}{{a,a,b}}{{c,d,d}}{{e,f,f}}{{e,c,b}} 
\tileix \& \tilex \& \tileix \\
\tilevii \& \tileix \& \tilex \\
\tileix \& \tilex \& \tileix \\
\end{vmosaic} &
\begin{vmosaic}[.8]{3}{3}{{a,b,b}}{{c,c,d}}{{e,f,f}}{{e,d,a}} 
\tileii \& \tilex \& \tileix \\
\tilex \& \tileix \& \tilex \\
\tileix \& \tilex \& \tileix \\
\end{vmosaic} \\
$8_{10}$ & $8_{11}$ & $8_{12}$ & $8_{13}$
\end{tabular}
}

\medskip\noindent
\resizebox{\textwidth}{!}{%
\begin{tabular}{cccc}
\begin{vmosaic}[.8]{3}{3}{{a,b,b}}{{c,d,d}}{{e,f,f}}{{e,c,a}} 
\tileii \& \tilex \& \tileix \\
\tilex \& \tileix \& \tilex \\
\tileix \& \tilex \& \tileix \\
\end{vmosaic} &
\begin{vmosaic}[.8]{3}{3}{{a,b,c}}{{d,d,e}}{{f,f,e}}{{c,b,a}} 
\tileii \& \tilex \& \tileix \\
\tilex \& \tileix \& \tilex \\
\tileix \& \tilex \& \tileix \\
\end{vmosaic} &
\begin{vmosaic}[.6]{4}{4}{{a,b,c,d}}{{d,e,e,f}}{{g,g,f,h}}{{h,c,b,a}} 
\tileo \& \tileiii \& \tileix \& \tilei \\
\tilei \& \tileii \& \tilex \& \tileix \\
\tileix \& \tilex \& \tileix \& \tilex \\
\tileiii \& \tileix \& \tileiv \& \tileiii \\
\end{vmosaic} &
\begin{vmosaic}[.6]{4}{4}{{a,b,c,c}}{{b,d,d,e}}{{f,f,g,g}}{{e,a,h,h}} 
\tileiii \& \tilex \& \tilei \& \tileii \\
\tileii \& \tileix \& \tilex \& \tileix \\
\tileix \& \tilex \& \tileix \& \tilex \\
\tileiv \& \tileiii \& \tileiv \& \tileiii \\
\end{vmosaic} \\
$8_{14}$ & $8_{15}$ & $8_{16}$ & $8_{17}$
\end{tabular}
}

\medskip\noindent
\resizebox{\textwidth}{!}{%
\begin{tabular}{cccc}
\begin{vmosaic}[.6]{4}{4}{{a,b,c,c}}{{b,d,e,e}}{{d,f,g,g}}{{f,a,h,h}} 
\tileiii \& \tileix \& \tilei \& \tileii \\
\tileii \& \tilex \& \tileix \& \tilex \\
\tilex \& \tileix \& \tilex \& \tileiv \\
\tileiv \& \tileiii \& \tileix \& \tilei \\
\end{vmosaic} &
\begin{vmosaic}[.8]{3}{3}{{a,a,b}}{{c,d,d}}{{e,e,c}}{{b,f,f}} 
\tileix \& \tilex \& \tileix \\
\tilex \& \tilex \& \tilex \\
\tileix \& \tilex \& \tileix \\
\end{vmosaic} &
\begin{vmosaic}[.8]{3}{3}{{a,a,b}}{{c,d,e}}{{e,d,c}}{{f,f,b}} 
\tilex \& \tileix \& \tilex \\
\tilex \& \tileix \& \tilex \\
\tileix \& \tilex \& \tileiv \\
\end{vmosaic} &
\begin{vmosaic}[.8]{3}{3}{{a,a,b}}{{c,c,b}}{{d,e,f}}{{f,e,d}} 
\tilex \& \tileix \& \tileix \\
\tilex \& \tileix \& \tilex \\
\tileiii \& \tilex \& \tileix \\
\end{vmosaic} \\
$8_{18}$ & $8_{19}$ & $8_{20}$ & $8_{21}$
\end{tabular}
}

\section{Virtual Knots with Virtual Mosaic Number Two}\label{minimal-virtual-mosaics}
\noindent
\resizebox{\textwidth}{!}{%
\begin{tabular}{cccc}
\begin{vmosaic}[1.2]{2}{2}{{a,b}}{{b,a}}{{c,d}}{{c,d}} 
\tilex \& \tilei \\
\tileix \& \tileviii \\
\end{vmosaic} &
\begin{vmosaic}[1.2]{2}{2}{{a,b}}{{c,d}}{{a,b}}{{c,d}} 
\tilex \& \tileviii \\
\tilex \& \tileix \\
\end{vmosaic} &
\begin{vmosaic}[1.2]{2}{2}{{a,b}}{{c,a}}{{d,b}}{{c,d}} 
\tilex \& \tileviii \\
\tilex \& \tilex \\
\end{vmosaic} &
\begin{vmosaic}[1.2]{2}{2}{{a,b}}{{c,d}}{{a,b}}{{c,d}} 
\tilex \& \tileviii \\
\tilex \& \tilex \\
\end{vmosaic} \\
$2.1,g=1$ & $3.1, g=2$ & $3.2, g=1$ & $3.3, g=2$ 
\end{tabular}
}

\medskip\noindent
\resizebox{\textwidth}{!}{%
\begin{tabular}{cccc}
\begin{vmosaic}[1.2]{2}{2}{{a,b}}{{c,a}}{{d,b}}{{d,c}} 
\tilex \& \tileviii \\
\tilex \& \tilex \\
\end{vmosaic} &
\begin{vmosaic}[1.2]{2}{2}{{a,a}}{{b,c}}{{d,b}}{{c,d}} 
\tilex \& \tileix \\
\tilex \& \tilevii \\
\end{vmosaic} &
\begin{vmosaic}[1.2]{2}{2}{{a,b}}{{b,a}}{{c,c}}{{d,d}} 
\tilex \& \tilei \\
\tileix \& \tilex \\
\end{vmosaic} &
\begin{vmosaic}[1.2]{2}{2}{{a,b}}{{c,a}}{{b,c}}{{d,d}} 
\tilex \& \tileix \\
\tileix \& \tilevii \\
\end{vmosaic} \\
$3.4, g=2$ & $3.5, g=1$ & $3.6, g=0$ & $3.7, g=1$ 
\end{tabular}
}

\medskip\noindent
\resizebox{\textwidth}{!}{%
\begin{tabular}{cccc}
\begin{vmosaic}[1.2]{2}{2}{{a,b}}{{c,b}}{{c,d}}{{a,d}} 
\tilex \& \tileix \\
\tileix \& \tilex \\
\end{vmosaic} &
\begin{vmosaic}[1.2]{2}{2}{{a,b}}{{a,c}}{{d,c}}{{d,b}} 
\tilex \& \tilex \\
\tileix \& \tilex \\
\end{vmosaic} &
\begin{vmosaic}[1.2]{2}{2}{{a,b}}{{c,b}}{{c,d}}{{a,d}} 
\tilex \& \tilex \\
\tileix \& \tileix \\
\end{vmosaic} &
\begin{vmosaic}[1.2]{2}{2}{{a,b}}{{a,c}}{{d,b}}{{d,c}} 
\tilex \& \tilex \\
\tileix \& \tileix \\
\end{vmosaic} \\
$4.1, g=2$ & $4.4, g=2$ & $4.8, g=2$ & $4.12, g=1$ 
\end{tabular}
}

\medskip\noindent
\resizebox{\textwidth}{!}{%
\begin{tabular}{cccc}
\begin{vmosaic}[1.2]{2}{2}{{a,b}}{{a,c}}{{d,c}}{{b,d}} 
\tilex \& \tileix \\
\tileix \& \tilex \\
\end{vmosaic} &
\begin{vmosaic}[1.2]{2}{2}{{a,b}}{{c,a}}{{c,d}}{{b,d}} 
\tilex \& \tilex \\
\tileix \& \tilex \\
\end{vmosaic} &
\begin{vmosaic}[1.2]{2}{2}{{a,b}}{{a,c}}{{d,c}}{{b,d}} 
\tilex \& \tilex \\
\tileix \& \tilex \\
\end{vmosaic} &
\begin{vmosaic}[1.2]{2}{2}{{a,b}}{{a,c}}{{d,d}}{{b,c}} 
\tilex \& \tilex \\
\tileix \& \tilex \\
\end{vmosaic} \\
$4.14, g=2$ & $4.21,g=2$ & $4.30,g=2$ & $4.36,g=1$ 
\end{tabular}
}

\medskip\noindent
\resizebox{\textwidth}{!}{%
\begin{tabular}{cccc}
\begin{vmosaic}[1.2]{2}{2}{{a,b}}{{a,c}}{{d,b}}{{d,c}} 
\tilex \& \tilex \\
\tilex \& \tilex \\
\end{vmosaic} &
\begin{vmosaic}[1.2]{2}{2}{{a,a}}{{b,b}}{{c,d}}{{c,d}} 
\tilex \& \tileix \\
\tileix \& \tilex \\
\end{vmosaic} &
\begin{vmosaic}[1.2]{2}{2}{{a,b}}{{a,c}}{{d,c}}{{b,d}} 
\tilex \& \tilex \\
\tileix \& \tileix \\
\end{vmosaic} &
\begin{vmosaic}[1.2]{2}{2}{{a,b}}{{a,c}}{{d,c}}{{d,b}} 
\tilex \& \tilex \\
\tileix \& \tileix \\
\end{vmosaic} \\
$4.37, g=1$ & $4.43, g=1$ & $4.48, g=2$ & $4.55, g=2$ 
\end{tabular}
}

\medskip\noindent
\resizebox{\textwidth}{!}{%
\begin{tabular}{cccc}
\begin{vmosaic}[1.2]{2}{2}{{a,b}}{{c,a}}{{c,d}}{{b,d}} 
\tilex \& \tilex \\
\tileix \& \tileix \\
\end{vmosaic} &
\begin{vmosaic}[1.2]{2}{2}{{a,a}}{{b,c}}{{d,b}}{{d,c}} 
\tilex \& \tileix \\
\tileix \& \tilex \\
\end{vmosaic} &
\begin{vmosaic}[1.2]{2}{2}{{a,b}}{{c,a}}{{c,b}}{{d,d}} 
\tilex \& \tilex \\
\tileix \& \tilex \\
\end{vmosaic} &
\begin{vmosaic}[1.2]{2}{2}{{a,b}}{{a,c}}{{d,c}}{{b,d}} 
\tilex \& \tilex \\
\tilex \& \tilex \\
\end{vmosaic} \\
$4.59, g=2$ & $4.64, g=1$ & $4.65, g=1$ & $4.71, g=2$ 
\end{tabular}
}

\medskip\noindent
\resizebox{\textwidth}{!}{%
\begin{tabular}{cccc}
\begin{vmosaic}[1.2]{2}{2}{{a,b}}{{a,c}}{{d,c}}{{d,b}} 
\tilex \& \tilex \\
\tilex \& \tilex \\
\end{vmosaic} &
\begin{vmosaic}[1.2]{2}{2}{{a,b}}{{c,a}}{{d,b}}{{c,d}} 
\tilex \& \tilex \\
\tilex \& \tilex \\
\end{vmosaic} &
\begin{vmosaic}[1.2]{2}{2}{{a,b}}{{c,a}}{{d,b}}{{c,d}} 
\tilex \& \tilex \\
\tileix \& \tilex \\
\end{vmosaic} &
\begin{vmosaic}[1.2]{2}{2}{{a,b}}{{c,a}}{{d,c}}{{b,d}} 
\tilex \& \tilex \\
\tilex \& \tilex \\
\end{vmosaic} \\
$4.77, g=2$ & $4.92, g=1$ & $4.95, g=1$ & $4.99, g=1$ 
\end{tabular}
}

\medskip\noindent
\resizebox{\textwidth}{!}{%
\begin{tabular}{cccc}
\begin{vmosaic}[1.2]{2}{2}{{a,b}}{{c,a}}{{d,b}}{{c,d}} 
\tilex \& \tileix \\
\tileix \& \tilex \\
\end{vmosaic} &
\begin{vmosaic}[1.2]{2}{2}{{a,b}}{{c,a}}{{d,c}}{{b,d}} 
\tilex \& \tileix \\
\tileix \& \tilex \\
\end{vmosaic} &
\phantom{\begin{vmosaic}[2.4]{1}{1}{{a}}{{a}}{{b}}{{b}} 
\tileo \\
\end{vmosaic}} &
\phantom{\begin{vmosaic}[2.4]{1}{1}{{a}}{{a}}{{b}}{{b}} 
\tileo \\
\end{vmosaic}} \\
$4.104, g=1$ & $4.105, g=1$ & & 
\end{tabular}
}

\section*{Acknowledgements}

We would like to thank Kyle Miller for providing us with the KnotFolio and Virtual KnotFolio tools that were so useful in conducting this research. The authors would also like to thank the Simons Foundation (\#426566, Allison Henrich) for their support of this research.


\end{document}